\documentclass{amsart}

\usepackage{amsmath,amssymb}
\usepackage{mathrsfs}
\usepackage{mathtools}
\usepackage{stmaryrd}
\usepackage{enumerate}
\usepackage{tikz-cd}
\usepackage[all]{xy}
\usepackage{aliascnt}
\usepackage[colorlinks=true, linkcolor=blue, citecolor=blue]{hyperref}

\usepackage{enumerate}

\usepackage{combelow}

\newtheorem{maintheorem}{Theorem}

\newaliascnt{maincorollary}{maintheorem}
\newtheorem{maincorollary}[maincorollary]{Corollary}
\aliascntresetthe{maincorollary}

\newtheorem{theorem}{Theorem}

\newaliascnt{lemma}{theorem}
\newtheorem{lemma}[lemma]{Lemma}
\aliascntresetthe{lemma}

\newaliascnt{corollary}{theorem}
\newtheorem{corollary}[corollary]{Corollary}
\aliascntresetthe{corollary}

\newaliascnt{proposition}{theorem}
\newtheorem{proposition}[proposition]{Proposition}
\aliascntresetthe{proposition}

\newaliascnt{conjecture}{theorem}

\aliascntresetthe{conjecture}

\newaliascnt{question}{theorem}

\aliascntresetthe{question}

\theoremstyle{definition}

\newaliascnt{definition}{theorem}
\newtheorem{definition}[definition]{Definition}
\aliascntresetthe{definition}

\newaliascnt{remark}{theorem}
\newtheorem{remark}[remark]{Remark}
\aliascntresetthe{remark}

\newaliascnt{example}{theorem}

\aliascntresetthe{example}

\newaliascnt{notation}{theorem}
\newtheorem{notation}[notation]{Notation}
\aliascntresetthe{notation}

\newtheorem*{acknowledgements}{Acknowledgements}

\newif\ifhascomments \hascommentstrue
\ifhascomments
  \newcommand{\kent}[1]{{\color{red}[[\ensuremath{\heartsuit\heartsuit\heartsuit} #1]]}}
  \newcommand{\matt}[1]{{\color{red}[[\ensuremath{\spadesuit\spadesuit\spadesuit} #1]]}}
  \newcommand{\jeremy}[1]{{\color{red}[[\ensuremath{\clubsuit\clubsuit\clubsuit} #1]]}}
\else
  \newcommand{\kent}[1]{}
  \newcommand{\matt}[1]{}
  \newcommand{\jeremy}[1]{}
\fi

\renewcommand{\setminus}{\smallsetminus}

\newcommand{\Z}{\mathbb{Z}}
\newcommand{\bZ}{\mathbb{Z}}

\newcommand{\QQ}{\mathbb{Q}}
\newcommand{\Q}{\mathbb{Q}}
\newcommand{\C}{\mathbb{C}}
\newcommand{\R}{\mathbb{R}}

\newcommand{\cX}{\mathcal{X}}
\newcommand{\cY}{\mathcal{Y}}
\newcommand{\cZ}{\mathcal{Z}}
\newcommand{\cW}{\mathcal{W}}
\newcommand{\cS}{\mathcal{S}}

\newcommand{\cA}{\mathcal{A}}
\newcommand{\cC}{\mathcal{C}}
\newcommand{\cD}{\mathcal{D}}
\newcommand{\cU}{\mathcal{U}}
\newcommand{\cO}{\mathcal{O}}
\newcommand{\cV}{\mathcal{V}}
\newcommand{\cF}{\mathcal{F}}
\newcommand{\cG}{\mathcal{G}}
\newcommand{\cH}{\mathcal{H}}
\newcommand{\cJ}{\mathscr{J}}

\newcommand{\cK}{\mathcal{K}}
\newcommand{\cE}{\mathcal{E}}
\newcommand{\cP}{\mathcal{P}}

\newcommand{\sL}{\mathscr{L}}
\newcommand{\sM}{\mathscr{M}}

\newcommand{\sJ}{\mathscr{J}}

\newcommand{\bL}{\mathbb{L}}

\newcommand{\bG}{\mathbb{G}}
\newcommand{\bA}{\mathbb{A}}
\newcommand{\bP}{\mathbb{P}}

\newcommand{\g}{\mathfrak{g}}

\newcommand{\Var}{\mathbf{Var}}
\newcommand{\Stack}{\mathbf{Stack}}
\newcommand{\MHS}{\mathbf{pMHS}}
\newcommand{\MHM}{\mathbf{MHM}}

\newcommand{\diff}{\mathrm{d}}
\newcommand{\red}{\mathrm{red}}
\newcommand{\id}{\mathrm{id}}

\newcommand{\str}{\mathrm{str}}
\newcommand{\Hdg}{\mathrm{Hdg}}
\newcommand{\even}{\mathrm{even}}
\newcommand{\odd}{\mathrm{odd}}
\newcommand{\orb}{\mathrm{orb}}

\newcommand{\lPuiseux}{\{\!\{}
\newcommand{\rPuiseux}{\}\!\}}

\DeclareMathOperator{\ev}{ev}
\DeclareMathOperator{\inj}{inj}
\DeclareMathOperator{\rep}{rep}

\DeclareMathOperator{\val}{val}

\DeclareMathOperator{\e}{e}
\DeclareMathOperator{\Spec}{Spec}
\DeclareMathOperator{\Ext}{Ext}

\DeclareMathOperator{\Hom}{Hom}

\DeclareMathOperator{\GL}{GL}

\DeclareMathOperator{\im}{im}

\DeclareMathOperator{\uHom}{\underline{\Hom}}

\DeclareMathOperator{\HD}{HD}
\DeclareMathOperator{\wt}{wt}
\DeclareMathOperator{\shft}{shft}
\DeclareMathOperator{\Bl}{Bl}
\DeclareMathOperator{\Mat}{Mat}

\newcommand{\LL}{\bL}
\newcommand{\Y}{\cY}
\newcommand{\X}{\cX}
\newcommand{\J}{\cJ}
\newcommand{\A}{\cA}
\newcommand{\GG}{\bG}
\newcommand{\CC}{\mathbb{C}}
\newcommand{\PP}{\bP}
\newcommand{\ZZ}{\bZ}
\def\seq{\subseteq}
\newcommand\mt[1]{\begin{pmatrix}#1\end{pmatrix}}
\newcommand\equa[1]{\begin{equation}\begin{split}#1\end{split}\nonumber\end{equation}}
\def\f{\varphi}
\def\fg{\mathfrak{g}}
\newcommand\floor[1]{\lfloor#1\rfloor}
\renewcommand{\AA}{\bA}
\renewcommand{\L}{\sL}

\tikzset{cong/.style={draw=none,edge node={node [sloped, allow upside down, auto=false]{$\cong$}}},
         Isom/.style={above,every to/.append style={edge node={node [sloped, allow upside down, auto=false]{$\sim$}}}}}

\title{A cohomological interpretation for stringy Hodge numbers}

\author{Jiahui Huang, Matthew Satriano, and Jeremy Usatine}

\thanks{MS was partially supported by an NSERC Discovery Grant. JU was partially supported by the Simons Foundation Travel Support for Mathematicians program and NSF DMS-2502347.}

\address{Jiahui Huang, Department of Pure Mathematics, University of Waterloo}
\email{j346huan@uwaterloo.ca}

\address{Matthew Satriano, Department of Pure Mathematics, University of Waterloo}
\email{msatriano@uwaterloo.ca}

\address{Jeremy Usatine, Department of Mathematics, Florida State University}
\email{jusatine@fsu.edu}

\begin{document}

\begin{abstract}
We obtain a cohomological interpretation for Batyrev's stringy Hodge numbers in the full generality in which they are defined. In a previous paper, the second and third authors used motivic integration to define the stringy Hodge--Deligne invariant of a smooth Artin stack $\mathcal{X}$ and proved that when $\mathcal{X}$ is a crepant resolution of a variety $Y$ with log-terminal singularities, the generating function for the stringy Hodge numbers of $Y$ is equal to the stringy Hodge--Deligne invariant of $\mathcal{X}$. In this paper, we introduce a cohomology theory $H_{\mathrm{str}}^*(\mathcal{X})$ that computes the stringy Hodge--Deligne invariant of $\mathcal{X}$. Since, by previous work of the second and third authors, all varieties with log-terminal singularities admit a crepant resolution by an Artin stack, this gives a cohomological interpretation for stringy Hodge numbers of any variety with log-terminal singularities. We also show that in the special case where $\mathcal{X}$ is Deligne--Mumford, $H_{\mathrm{str}}^*(\mathcal{X})$ coincides with the orbifold cohomology of $\mathcal{X}$.
\end{abstract}

\maketitle

\setcounter{tocdepth}{1}

\tableofcontents

\section{Introduction}

\numberwithin{theorem}{section}
\numberwithin{lemma}{section}
\numberwithin{corollary}{section}
\numberwithin{proposition}{section}
\numberwithin{conjecture}{section}
\numberwithin{question}{section}
\numberwithin{remark}{section}
\numberwithin{definition}{section}
\numberwithin{example}{section}
\numberwithin{notation}{section}

The \emph{topological mirror symmetry test} predicts that if $Y, Y^\vee$ is a mirror pair of $d$-dimensional smooth Calabi--Yau varieties, then
\[
	h^{p,q}(Y) = h^{d-p, q}(Y^\vee)
\]
for all $p,q$. However in many of the known constructions for mirror pairs, $Y$ and $Y^\vee$ are not both smooth, and in that case the topological mirror symmetry test as stated does not work. This problem can be fixed when $Y$ and $Y^\vee$ admit crepant resolutions, as then their crepant resolutions form a mirror pair of smooth Calabi--Yau varieties to which the topological mirror symmetry test then may be applied. Unfortunately, the condition of admitting a crepant resolution is a rather restrictive one. For example if $Y$ is singular, $\Q$-factorial, and terminal, then $Y$ does not admit a crepant resolution. Nonetheless, Batyrev found a beautiful solution to this problem, namely to still consider singular Calabi--Yau varieties but to replace the usual Hodge numbers with his notion of \emph{stringy Hodge numbers}. 

Let $Y$ be a complex variety with log-terminal singularities. In \cite{Batyrev}, Batyrev introduced the \emph{stringy Hodge--Deligne invariant}\footnote{In Batyrev's original paper and often elsewhere in the literature, the stringy Hodge--Deligne invariant is referred to as the stringy $E$-function.} $\HD_{\str}(Y)$ of $Y$ and called its (sign-corrected) coefficients the stringy Hodge numbers of $Y$. Many examples have since demonstrated the utility of using these stringy Hodge numbers in order to formulate the topological mirror symmetry test, and more generally, these stringy Hodge numbers have become a well studied invariant in their own right. However, there remained a desire to find a cohomology theory that computes them (see e.g., \cite{BatyrevDais, Borisov}). To our knowledge, this problem of finding such a cohomology theory has remained wide open. In this paper, we propose a solution to this problem in full generality. In other words, we define such a cohomology theory, show that it indeed computes Batyrev's stringy Hodge numbers, and furthermore, we show that orbifold cohomology arises as a special case of our cohomology theory.

More precisely, by a ``cohomology theory that computes the stringy Hodge numbers of $Y$'', we mean a cohomology theory, endowed with a mixed Hodge structure, whose Hodge--Deligne invariant coincides with $\HD_{\str}(Y)$. In order to motivate our construction, we first mention some special cases where such a cohomology theory is known to exist. Although Batyrev's original definition for $\HD_{\str}(Y)$ does not describe it as the Hodge--Deligne invariant of some cohomology theory, the definition does immediately imply that some special cases have such a description. When $Y$ is smooth, $\HD_{\str}(Y)$ coincides with the usual Hodge--Deligne polynomial of $Y$ and therefore is the Hodge--Deligne invariant of the compactly supported cohomology of $Y$. More generally, and importantly from the perspective of mirror symmetry, if $Y$ admits a crepant resolution of singularities $X \to Y$, then $\HD_{\str}(Y)$ coincides with the Hodge--Deligne polynomial of $X$ and is therefore the Hodge--Deligne invariant of the compactly supported cohomology of $X$. However as we have mentioned, the condition of admitting a crepant resolution is a restrictive one, so these special cases are far from the general story.

In \cite{Yasuda} (see \cite{Yasuda2004} for a special case), Yasuda showed that if $\cX \to Y$ is a crepant resolution of singularities by a Deligne--Mumford stack $\cX$, then $\HD_{\str}(Y)$ coincides with the Hodge--Deligne invariant of the (compactly supported) orbifold cohomology of $\cX$ in the sense of Chen--Ruan \cite{ChenRuan}. Although we are not aware of any precise statements on whether or not admitting a crepant resolution by a Deligne--Mumford stack is rare, there are plenty of examples of log-terminal $Y$ that do not admit such a resolution. For example, if $Y$ is 1-Gorenstein and $\HD_{\str}(Y)$ is not a polynomial (see e.g., \cite{BatyrevMoreau, SchepersVeys} for such examples), then $Y$ does not admit a resolution of singularities by a Deligne--Mumford stack. Furthermore, there are examples where $Y$ is 1-Gorenstein and $\HD_{\str}(Y)$ is a polynomial but $Y$ nonetheless does not admit a crepant resolution of singularities by a Deligne--Mumford stack (see e.g., \cite[Remark 5.9]{vandenBergh}). In summary, in order to obtain crepant resolutions of singularities for all log-terminal varieties, Deligne--Mumford stacks are not sufficient.

In \cite{SatrianoUsatine3}, the second and third author showed that any log-terminal variety $Y$ admits a crepant resolution of singularities $\cX \to Y$ by an Artin stack. However at that time (to our knowledge), there was not even a proposal for a cohomology theory associated to $\cX$ whose Hodge--Deligne invariant coincides with $\HD_{\str}(Y)$. As an intermediate step in finding such a cohomology theory, the same authors in \cite{SatrianoUsatine6} used a certain motivic integral over twisted arcs of $\cX$ to define the \emph{stringy Hodge--Deligne invariant} $\HD_{\str}(\cX)$ of $\cX$ and showed that $\HD_\str(\cX) = \HD_{\str}(Y)$. The remaining goal was therefore to find a cohomology theory for $\cX$ whose Hodge--Deligne invariant coincides with $\HD_\str(\cX)$. In this paper we accomplish that goal, and as a result, we obtain a cohomological interpretation for stringy Hodge numbers \emph{in full generality}. 

Before we describe the construction, we highlight one appealing aspect to the approach we have taken. Not only do we obtain a cohomological interpretation for stringy Hodge numbers, but we do so by returning to an idea deeply entrenched in the historical mirror symmetry literature: if one wants to work with a singular Calabi--Yau variety, one should replace it with a crepant resolution.

\subsection{Main results}

We begin by recalling the definition of the stringy Hodge--Deligne invariant of an Artin stack. For the meaning of the motivic integral $\int_{\cA_\cX} \bL^{-\wt_\cX} \diff\nu_\cX \in \widehat{\sM}_\C$ and the ring homomorphism $\HD: \widehat{\sM}_\C: \to \Z \lPuiseux u^{-1}, v^{-1} \rPuiseux$, we refer to \autoref{preliminariesMotivicIntegration} and \autoref{preliminariesGrothendieckRings}, respectively, below.

\begin{definition}
Let $\cX$ be a smooth finite type equidimensional Artin stack over $\C$ with affine diagonal and properly stable good moduli space. Then $\bL^{-\wt_\cX}$ is integrable on $\cA_\cX$ by \autoref{weightFunctionIsIntegrable} below, so we may define the \emph{stringy class} of $\cX$ to be
\[
	\e_{\str}(\cX) = \bL^{\dim\cX} \int_{\cA_\cX} \bL^{-\wt_\cX} \diff\nu_\cX,
\]
and we may define the \emph{stringy Hodge--Deligne invariant} of $\cX$ to be
\[
	\HD_{\str}(\cX) = \HD(\e_{\str}(\cX)) \in \Z \lPuiseux u^{-1}, v^{-1} \rPuiseux.
\]
\end{definition}

We will now set some notation that will be used to define our desired cohomology theory. We refer to \autoref{preliminariesMotivicIntegration} below for the definitions of $\cA_\cX$, $\wt_\cX$, and $\sJ_n(\cX)$.

\begin{notation}
If $\cX$ is a smooth finite type equidimensional Artin stack over $\C$ with affine diagonal and properly stable good moduli space, then for any $n \in \Z_{\geq 0}$ and $w \in \Q$, we will let $\cA_{\cX, n}^w$ denote the image of $\cA_\cX \cap \wt_\cX^{-1}(w)$ in $|\sJ_n(\cX)|$. Then by \autoref{theoremStringyClassWellDefined}(\ref{CanonicalSetOfTwistedArcsHasConstructibleImage}) below, each $\cA^w_{\cX,n}$ is a quasi-compact locally constructible subset. Therefore we can canonically associate to $\cA^w_{\cX,n}$ an Artin stack $\cY^{\cA^w_{\cX,n}}$ over $\C$ as described in \autoref{notationForCanonicalDecompostion} below, and we will set $\cY_{\cX, n}^w = \cY^{\cA^w_{\cX,n}}$. In particular, $\cY_{\cX,n}^w$ is finite type over $\C$ by \autoref{propositionThatCanonicalDecompositionIsDecomposition} below and is a disjoint union of locally closed substacks of $\sJ_n(\cX)$ whose supports partition $\cA^w_{\cX,n}$. 
\end{notation}

\begin{remark}\label{rmk:constructibleMHS}
The main reason we consider the stacks $\cY_{\cX,n}^w$, and more specifically their compactly supported cohomology, is in order to canonically assign mixed Hodge structures to the constructible subsets $\cA^w_{\cX, n}$. We expect that if one instead considers an alternative ``nice'' way of canonically associating mixed Hodge structures to constructible subsets, then the ideas in this paper will still apply and thus one will still obtain a cohomological interpretation for stringy Hodge numbers.
\end{remark}

The following is the main structural theorem that will allow us to define our cohomology theory. See \autoref{definitionOfStackCohomology} below for the definition of compactly supported cohomology, with mixed Hodge structure, of a finite type Artin stack over $\C$.

\begin{maintheorem}\label{stringyCohomologyStabilizes}
Let $\cX$ be a smooth finite type equidimensional Artin stack over $\C$ with affine diagonal and properly stable good moduli space, let $w \in \Q$, and let $i \in \Z$. Then there exists some $m \in \Z_{\geq 0}$ such that for all $n \geq m$,
\[
	H_c^{i+2n\dim\cX}(\cY_{\cX,n}^w)(n\dim\cX) \cong H_c^{i+2m\dim\cX}(\cY_{\cX, m}^w)(m\dim\cX).
\]
\end{maintheorem}

\autoref{stringyCohomologyStabilizes} allows us to make this paper's main definition.

\begin{definition}
Let $\cX$ be a smooth finite type equidimensional Artin stack over $\C$ with affine diagonal and properly stable good moduli space. For each $i \in \Z$ and $w \in \Q$, we define
\[
	H_{\str}^{i,w}(\cX) = H_c^{i+2m\dim\cX}(\cY_{\cX, m}^w)(m\dim\cX),
\]
where $m$ satisfies the conclusion of \autoref{stringyCohomologyStabilizes}. Then $H_{\str}^{i,w}(\cX)$ is a polarizable mixed Hodge structure and is well defined up to isomorphism by \autoref{stringyCohomologyStabilizes}. Furthermore, for each $i \in \Q$ we define
\begin{align*}
	&H_{\str}^{\even, i}(\cX) = \bigoplus_{w \in \Q,\text{ $i+2w$ is an even integer}} H_{\str}^{i+2w, w}(\cX)(w),\\
	&H_{\str}^{\odd, i}(\cX) =  \bigoplus_{w \in \Q,\text{ $i+2w$ is an odd integer}} H_{\str}^{i+2w, w}(\cX)(w),
\end{align*}
and
\[
	H_{\str}^i(\cX) = H_{\str}^{\even, i}(\cX) \oplus H_{\str}^{\odd, i}(\cX).
\]
\end{definition}

\begin{remark}
Since $\cX$ has a properly stable good moduli space, the function $\wt_\cX$ takes only finitely many values on $\cA_\cX$ by \cite[Proposition 4.19]{SatrianoUsatine6} and the fact that the canonical reduction of stabilizers of $\cX$, in the sense of \cite{EdidinRydh}, is tame, so the direct sums defining $H_{\str}^{\even, i}(\cX)$ and $H_{\str}^{\odd, i}(\cX)$ are finite.
\end{remark}

\begin{remark}[{Computational aspects of $H_{\str}^*(\cX)$}]
Due to the algorithmic nature of the Edidin--Rydh construction \cite[Theorem 2.11]{EdidinRydh}, the stringy cohomology of $\cX$ is readily computable. We demonstrate this in \autoref{prop:HstrA1-3d} where we compute the stringy cohomology of the canonical toric stack living over the threefold $A_1$-singularity.
\end{remark}

For the next remark, we refer to \autoref{preliminariesGrothendieckRings} below for the definition of the ring $\Z \llparenthesis u^{-1}, v^{-1} \rrparenthesis$.

\begin{remark}\label{remarkWeighFunctionIntegral}
If $\HD_{\str}(\cX) \in \Z \llparenthesis u^{-1}, v^{-1} \rrparenthesis$, then $\wt_\cX(\cA_\cX) \subset \Z$ by \autoref{propositionIntegralWeightFunction} below. Therefore $H_{\str}^{i,w}(\cX) = 0$ unless $w \in \Z$, so $H_{\str}^{\even, i}(\cX) = 0$ unless $i$ is an even integer, and $H_{\str}^{\odd, i}(\cX) = 0$ unless $i$ is an odd integer. Note that if $\cX$ is a crepant resolution of a $\Q$-Gorenstein irreducible finite type scheme over $\C$ whose stringy Hodge--Deligne invariant is contained in $\Z \llparenthesis u^{-1}, v^{-1} \rrparenthesis$, then $\HD_{\str}(\cX) \in \Z \llparenthesis u^{-1}, v^{-1} \rrparenthesis$ by \cite[Corollary 1.9]{SatrianoUsatine6}, so the previous statement applies.
\end{remark}

Our next main theorem shows that our cohomology theory does indeed compute the stringy Hodge--Deligne invariant of an Artin stack. We refer to \autoref{preliminariesGrothendieckRings} below for the definition of the ring homomorphism $\chi_{\Hdg}: \widehat{\sM}_\C \to \widehat{K_0(\MHS)}$.

\begin{maintheorem}\label{theoremStringyCohomologyGivesStringyHD}
Let $\cX$ be a smooth finite type equidimensional Artin stack over $\C$ with affine diagonal and properly stable good moduli space. Then for each $w \in \Q$,
\[
	\chi_{\Hdg}(\bL^{\dim\cX}\nu_\cX(\cA_\cX \cap \wt_\cX^{-1}(w))) = \sum_{i \in \Z} (-1)^{i}[H_{\str}^{i,w}(\cX)].
\]
In particular,
\[
	\HD_{\str}(\cX) = \sum_{i \in \Q} \left(\HD(H_{\str}^{\even, i}(\cX)) -  \HD(H_{\str}^{\odd, i}(\cX))\right).
\]
Furthermore if $\HD_{\str}(\cX) \in \Z \llparenthesis u^{-1}, v^{-1} \rrparenthesis$, then
\[
	\chi_{\Hdg}(\e_{\str}(\cX)) = \sum_{i \in \Z} (-1)^i [H^i_{\str}(\cX)],
\]
and in particular
\[
	\HD_{\str}(\cX) = \sum_{i \in \Z}(-1)^i \HD(H_{\str}^i(\cX)).
\]
\end{maintheorem}

The following corollary, which immediately follows from \autoref{remarkWeighFunctionIntegral}, \autoref{theoremStringyCohomologyGivesStringyHD}, and \cite[Corollary 1.9 and Proposition 12.3]{SatrianoUsatine6}, provides a cohomological interpretation for stringy Hodge numbers in full generality. 

\begin{maincorollary}\label{corollaryStringyCohomologyCrepantResolution}
Let $Y$ be a log-terminal irreducible finite type scheme over $\C$, and let $\cX \to Y$ be a crepant resolution with $\cX$ a smooth finite type irreducible Artin stack over $\C$ with affine diagonal (such a resolution exists by \cite[Theorem 1.1]{SatrianoUsatine3}, or see \cite[Theorem 1.2]{SatrianoUsatine6} for the statement in this level of generality). Then
\[
	\HD_{\str}(Y) = \sum_{i \in \Q} \left(\HD(H_{\str}^{\even, i}(\cX)) -  \HD(H_{\str}^{\odd, i}(\cX))\right).
\]
Furthermore\footnote{In the original definition for stringy Hodge numbers \cite[Definition 3.8]{Batyrev}, Batyrev assumed the stronger condition that the stringy Hodge--Deligne invariant is contained in $\Z[u,v]$.} if $\HD_{\str}(Y) \in \Z \llparenthesis u^{-1}, v^{-1} \rrparenthesis$, for example if $Y$ is 1-Gorenstein, then
\[
	\HD_{\str}(Y) = \sum_{i \in \Z}(-1)^i \HD(H_{\str}^i(\cX)).
\]
\end{maincorollary}

\begin{remark}
We note that every log-terminal variety $Y$ over $\C$ admits many different choices of crepant resolutions $\cX\to Y$; in fact, \cite[Theorem 1.1]{SatrianoUsatine3} produces infinitely many such choices. Thus, \autoref{corollaryStringyCohomologyCrepantResolution} gives many different cohmological interpretations of $\HD_{\str}(Y)$. This is not a bug but a feature. Indeed, as stated in \cite{Borisov}, ``One expects to have not one stringy cohomology space but a family of such spaces that interpolates between cohomology of different crepant (orbifold) resolutions.''
\end{remark}

Our final main result shows that in the special case where $\cX$ is Deligne--Mumford, our cohomology theory coincides with the orbifold cohomology of $\cX$. See \autoref{subsectionOrbifoldCohomologyPreliminaries} below for the definitions of $H_{\orb}^{\even,i}(\cX), H_{\orb}^{\odd,i}(\cX), H_{\orb}^i(\cX)$, and $\HD_{\orb}(\cX)$.

\begin{maintheorem}\label{theoremStringyIsOrbifold}
Let $\cX$ be a smooth finite type equidimensional tame Artin stack over $\C$ with affine diagonal. Then for all $i \in \Q$,
\[
	H_{\str}^{\even,i}(\cX) \cong H_{\orb}^{\even,i}(\cX), \qquad \text{and} \qquad H_{\str}^{\odd,i}(\cX) \cong H_{\orb}^{\odd,i}(\cX).
\]
In particular,
\[
	\HD_{\str}(\cX) = \HD_{\orb}(\cX),
\]
and for all $i \in \Q$,
\[
	H_{\str}^i(\cX) \cong H_{\orb}^i(\cX).
\]
\end{maintheorem}

\acknowledgements{We thank Takehiko Yasuda for the conversation that led to \autoref{YasudaShiftConvention}. We also thank Dan Abramovich, Donu Arapura, Victor Batyrev, Oishee Banerjee, Lev Borisov, and Dori Bejleri for helpful conversations.}

\section{Preliminaries}

For the remainder of this paper, we will let $k$ be an algebraically closed field of characteristic 0, and for any stack $\cY$, we will let $|\cY|$ denote its associated topological space.

\subsection{Grothendieck rings of varieties and stacks}\label{preliminariesGrothendieckRings}

We will set some notation that will be used throughout this paper. We will let $\widehat{\sM}_k$ denote the ring obtained from the Grothendieck ring of varieties over $k$ by inverting the class of $\bA^1_k$ and then completing with respect to the dimension filtration, we will let $\bL \in \widehat{\sM}_k$ denote the class of $\bA^1_k$, and we will let $\Vert \cdot \Vert: \widehat{\sM}_k \to \R_{\geq 0}$ denote the norm given by the dimension filtration, normalized so that the class of any finite type scheme $Y$ over $k$ has norm equal to $\exp(\dim Y)$. We will let $\widehat{K_0(\MHS)}$ denote the ring obtained from the Grothendieck ring of polarizable mixed Hodge structures by completing with respect to the filtration induced by weight of a mixed Hodge structure, and for any polarizable mixed Hodge structure $V$, we will let $[V]$ denote its class in $\widehat{K_0(\MHS)}$. We will let $\Vert \cdot \Vert: \widehat{K_0(\MHS)} \to \R_{\geq 0}$ denote the norm given by the filtration, normalized so that the class of any pure Hodge structure of weight $2d$ has norm equal to $\exp(d)$. We will let $\chi_{\Hdg}: \widehat{\sM}_\C \to \widehat{K_0(\MHS)}$ denote the unique continuous ring homomorphism that takes the class of any finite type scheme $Y$ over $\C$ to $\sum_{i \in \Z_{\geq 0}} (-1)^i[H^i_c(Y)]$. Note that for any $\Theta \in \widehat{\sM}_\C$, we have $\Vert \Theta \Vert \geq \Vert \chi_{\Hdg}(\Theta) \Vert$.

We will let $\Z\llparenthesis x, y \rrparenthesis$ denote the result of completing $\Z[x^{\pm 1}, y^{\pm 1}]$ with respect to the filtration induced by lowest total degree. Therefore elements of $\Z\llparenthesis x, y \rrparenthesis$ are of the form $\sum_{n, m \in \Z} a_{n,m}x^n y^m$ with $a_{n,m} \in \Z$ such that for any $d \in \Z$, we have $a_{n,m} = 0$ for all but finitely many $n,m$ with $n + m \leq d$. We will then set
\[
	\Z \lPuiseux x, y \rPuiseux = \bigcup_{m \in \Z_{>0}} \Z\llparenthesis x^{1/m}, y^{1/m} \rrparenthesis.
\]
For any polarizable mixed Hodge structure $V$, we will denote its Hodge--Deligne invariant by
\[
	\HD(V) = \sum_{p,q \in \Z} h^{p,q}(V)u^p v^q \in \Z[u^{\pm 1}, v^{\pm 1}],
\]
where $h^{p,q}(V)$ is the dimension of the $(p,q)$ piece of $V$. 
\begin{remark}
In the statements of \autoref{theoremStringyCohomologyGivesStringyHD} and \autoref{corollaryStringyCohomologyCrepantResolution}, $H^{\even,i}_{\str}(\cX)$ and $H^{\odd,i}_{\str}(\cX)$ may have pieces in fractional degree. For such a fractional mixed Hodge structure $V$, we set
\[
	\HD(V) = \sum_{p,q \in \Q} h^{p,q}(V)u^p v^q \in \bigcup_{m \in \Z_{>0}}\Z[u^{\pm 1/m}, v^{\pm 1/m}].
\]
\end{remark}

We then let $\HD: \widehat{\sM}_\C \to \Z\llparenthesis u^{-1}, v^{-1} \rrparenthesis$ denote the composition of $\chi_{\Hdg}$ with the unique continuous ring homomorphism $\widehat{K_0(\MHS)} \to \Z\llparenthesis u^{-1}, v^{-1} \rrparenthesis$ that takes the class of any polarizable mixed Hodge structure $V$ to $\HD(V)$. In particular, $\HD: \widehat{\sM}_\C \to \Z\llparenthesis u^{-1}, v^{-1} \rrparenthesis$ takes the class of any finite type scheme over $\C$ to its Hodge--Deligne polynomial. For any $m \in \Z_{>0}$ we will also let $\HD: \widehat{\sM}_\C[\bL^{1/m}] \to \Z \lPuiseux u^{-1}, v^{-1} \rPuiseux$ denote the ring homomorphism that agrees with $\HD: \widehat{\sM}_\C \to \Z\llparenthesis u^{-1}, v^{-1} \rrparenthesis$ on $\widehat{\sM}_\C$ and takes $\bL^{1/m}$ to $(uv)^{1/m}$.

We will let $K_0(\Stack_k)$ denote the Grothendieck ring of Artin stacks over $k$ in the sense of \cite{Ekedahl}. There is a unique ring homomorphism $K_0(\Stack_k) \to \widehat{\sM}_k$ that takes the class in $K_0(\Stack_k)$ of any finite type scheme over $k$ to its class in $\widehat{\sM}_k$. For any finite type stack $\cY$ over $k$ with affine geometric stabilizers, we will let $\e(\cY) \in \widehat{\sM}_k$ denote the image under $K_0(\Stack_k) \to \widehat{\sM}_k$ of the class of $\cY$ in $K_0(\Stack_k)$. We recall the following proposition that will be useful throughout this paper.

\begin{proposition}[{\cite[Proposition 4.4]{SatrianoUsatine5}}]\label{RecallingNormAndDimension}
If $\cY$ is a finite type Artin stack over $k$ with affine geometric stabilizers, then
\[
	 \Vert \e(\cY) \Vert = \exp(\dim\cY).
\]
\end{proposition}

\begin{remark}
In \autoref{RecallingNormAndDimension}, we are using the convention that $\dim\emptyset = -\infty$ and $\exp(-\infty) = 0$.
\end{remark}

\subsection{Motivic integration for Artin stacks}\label{preliminariesMotivicIntegration}

In this subsection, we recall the notation we will use for motivic integration over twisted arcs of Artin stacks. We refer to \cite{SatrianoUsatine6} for further details. Throughout this subsection, let $\cX$ be a finite type Artin stack over $k$ with affine diagonal.

For any $n \in \Z_{\geq 0}$ and $\ell \in \Z_{>0}$, we will let $\cD^\ell_{n}$ denote the stack quotient $[\Spec( k[t^{1/\ell}]/(t^{n+1}) ) / \mu_\ell]$, where $\xi \in \mu_\ell$ acts on $k[t^{1/\ell}]/(t^{n+1})$ by $f(t^{1/\ell}) \mapsto f(\xi t^{1/\ell})$. We then let
\[
	\sJ_n^\ell(\cX) = \uHom^{\rep}_k(\cD^\ell_{n}, \cX)
\]
be \emph{the stack of twisted $n$-jets of $\cX$ of order $\ell$}, which is a finite type Artin stack over $k$ with affine diagonal \cite[Remark 4.4]{SatrianoUsatine6}, and for any $m \leq n$, we let $\theta^n_m: \sJ_n^\ell(\cX) \to \sJ_m^\ell(\cX)$ denote the truncation map induced by the closed immersion $\cD^\ell_m \hookrightarrow \cD^\ell_n$. In the special case where $\ell = 1$, we will sometimes use the notation $\sL_n(\cX) = \sJ_n^1(\cX)$. We will also set
\[
	\sJ_n(\cX) = \bigsqcup_{\ell \in \Z_{>0}} \sJ_n^\ell(\cX),
\]
and each $\theta^n_m: \sJ_n(\cX) \to \sJ_m(\cX)$ will denote the obvious map. We will denote the \emph{stack of twisted arcs of $\cX$ of order $\ell$} by
\[
	\sJ^\ell(\cX) = \varprojlim_n \sJ_n^\ell(\cX)
\]
and let each $\theta_n: \sJ^\ell(\cX) \to \sJ_n(\cX)$ denote the canonical map. We will also set
\[
	\sJ(\cX) = \bigsqcup_{\ell \in \Z_{>0}} \sJ^\ell(\cX)
\]
and each $\theta_n: \sJ(\cX) \to \sJ_n(\cX)$ will denote the obvious map. We call a subset $\cC \subset |\sJ(\cX)|$ \emph{bounded} if $\cC \cap |\sJ^\ell(\cX)| = \emptyset$ for all but finitely many $\ell$. We call a subset $\cC \subset |\sJ(\cX)|$ a \emph{cylinder} if it is the preimage along $\theta_n$ of a locally constructible subset of $|\sJ_n(\cX)|$ for some $n$. 

For the remainder of this subsection, we will additionally assume that $\cX$ is equidimensional, smooth over $k$, and has a good moduli space. Then to each bounded cylinder $\cC \subset \sJ(\cX)$, there is an associated \emph{motivic volume} $\nu_\cX(\cC) \in \widehat{\sM}_k$ as defined in \cite[Definition 4.13]{SatrianoUsatine6}. We will also use the notion of \emph{measurable} subset $\cA \subset |\sJ(\cX)|$ and its \emph{motivic volume} $\nu_\cX(\cA)$. Measurable subsets are those subsets that are approximated by bounded cylinders in a precise sense, and we refer to \cite[Definition 4.23]{SatrianoUsatine6} for details. If $m \in \Z_{>0}$, $\cA \subset |\sJ(\cX)|$, and $f: \cA \to (1/m)\Z \cup \{\infty\}$ is a function such that $f^{-1}(w)$ is measurable for all $w \in (1/m)\Z$ and such that
\[
	\sum_{w \in (1/m)\Z} \bL^w \nu_\cX(f^{-1}(w))
\]
converges in $\widehat{\sM}_k[\bL^{1/m}]$, we say $\bL^f$ is \emph{integrable} (on $\cA$) and set
\[
	\int_\cA \bL^{f} \diff\nu_\cX = \sum_{w \in (1/m)\Z} \bL^w \nu_\cX(f^{-1}(w)).
\]

We will let $\wt_\cX: |\sJ(\cX)| \to \Q$ denote the \emph{weight function} as defined in \cite[Definition 4.20]{SatrianoUsatine6}. Finally when $\cX$ has a stable good moduli space, we will let $\cA_\cX \subset |\sJ(\cX)|$ denote the image of $|\sJ(\cX')|$ in $|\sJ(\cX)|$, where $\cX' \to \cX$ is the canonical reduction of stabilizers of Edidin--Rydh \cite{EdidinRydh}.

\subsection{Compactly supported cohomology of Artin stacks}

We will need a notion of compactly supported cohomology, endowed with a polarizable mixed Hodge structure, for finite type Artin stacks over $\C$. For this notion, we will use Tubach's theory of (the derived category of) mixed Hodge modules for Artin stacks. In this subsection, we will collect the necessary properties we will use from this theory. We claim no originality to the results in this subsection and emphasize that the statements here are all either explicitly stated in \cite{Tubach2} or, at the very least, are immediate consequences of results in loc. cit.

\begin{notation}
In \cite{Tubach1}, Tubach defines for any finite type scheme $X$ over $\C$ an $\infty$-categorical enhancement $D^b(\MHM(X))$ of Saito's bounded derived category of mixed Hodge modules, and in \cite{Tubach2}, extends the indization of $D^b(\MHM(-))$ (along with the accompanying six functor formalism) to Artin stacks. Following Tubach's notation, if $\cX$ is a locally finite type Artin stack over $\C$, we let $D_H(\cX)$ denote the resulting $\infty$-category, we let $D_{H,c}(\cX)$ denote the $\infty$-category of \emph{cohomologically constructible mixed Hodge modules} over $\cX$ as in \cite[Definition 3.5]{Tubach2}, and for any object $K$ of $D_H(\cX)$ and $n \in \Z$, we let $H^n(K)$ denote the $n$th cohomology object with respect to the $t$-structure in \cite[Proposition 3.4]{Tubach2}. Furthermore we let $\Q_\cX \in D_H(\cX)$ denote $s^*\Q$, where $s: \cX \to \Spec(\C)$ is the structure map. We will identify the homotopy category of the heart of $D^b(\MHM(\Spec(\C))$ with the category $\MHS$ of polarizable mixed Hodge structures using the canonical equivalence.
\end{notation}

\begin{remark}\label{remarkCohomologyObjectIsMixedHodgeStructure}
By the definition of $D_{H,c}(\Spec(\C))$, for any object $K \in D_{H,c}(\Spec(\C))$, each cohomology object $H^n(K)$ is in the heart of $D^b(\MHM(\Spec(\C)))$ and is therefore a polarizable mixed Hodge structure.
\end{remark}

\begin{definition}\label{definitionOfStackCohomology}
Let $\cX$ be a locally finite type Artin stack over $\C$, and let $i \in \Z$. The $i$th \emph{cohomology with rational coefficients} of $\cX$ is
\[
	H^i(\cX) = H^i(s_*\Q_\cX),
\]
and the $i$th \emph{compactly supported cohomology with rational coefficients} of $\cX$ is
\[
	H^i_c(\cX) = H^i(s_!\Q_\cX),
\]
where $s: \cX \to \Spec(\C)$ is the structure map. By \cite[Theorem 3.8]{Tubach2}, $s_*\Q_\cX$ and $s_!\Q_\cX$ are objects of $D_{H,c}(\Spec(\C))$, so $H^i(\cX)$ and $H^i_c(\cX)$ are polarizable mixed Hodge structures by \autoref{remarkCohomologyObjectIsMixedHodgeStructure}.
\end{definition}

\begin{proposition}\label{excisionExactSequence}
Let $\cX$ be a locally finite type Artin stack over $\C$, let $\cK \in D_H(\cX)$, let $i: \cZ \to \cX$ be a closed substack, and let $j: \cU \to \cX$ be the inclusion of the open complement of $\cZ$ in $\cX$. Then 
\[
	j_!j^!\cK \xrightarrow{\text{co-unit}} \cK \xrightarrow{\text{unit}} i_*i^*\cK
\]
is a distinguished triangle. In particular, there is a long exact sequence
\[
	\dots \to H_c^i(\cU) \to H_c^i(\cX) \to H_c^i(\cZ) \to H_c^{i+1}(\cU) \to \cdots.
\]
\end{proposition}

\begin{proof}
The first statement follows immediately from \cite[Theorem 3.1(5)]{Tubach2}. The second statement follows from the first by setting $\cK = \Q_\cX$ and using that $j^! \cong j^*$ by \cite[Theorem 3.1(7)]{Tubach2} and that $i_* \cong i_!$ by \cite[Theorem 3.1(4)]{Tubach2}.
\end{proof}

\begin{proposition}\label{cohomologyVanishingAboveTwiceDimension}
Let $\cX$ be a finite type Artin stack over $\C$. If $i > 2\dim\cX$, then $H_c^i(\cX) = 0$.
\end{proposition}

\begin{proof}
By \autoref{excisionExactSequence}, we may assume that $\cX$ is smooth and equidimensional. Then by Verdier duality \cite[Paragraph after Definition 3.11]{Tubach2} and the purity isomorphism \cite[Theorem 3.1(7)]{Tubach2}, it is sufficient to show that $H^i(\cX) = 0$ for all $i < 0$. But this follows from the definition of $D_H(\cX)$, which in particular gives that $H^i(\cX)$ is the $i$th cohomology of a simplicial scheme.
\end{proof}

\begin{proposition}\label{pqPieceVanishesAboveDegree}
Let $\cX$ be a finite type Artin stack over $\C$ with affine geometric stabilizers. If $p + q > i$, then $h^{p,q}(H_c^i(\cX)) = 0$.
\end{proposition}

\begin{proof}
This follows immediately from \cite[Proposition 3.22]{Tubach2}
\end{proof}

\subsection{Orbifold cohomology}\label{subsectionOrbifoldCohomologyPreliminaries}

In this subsection, we will recall the definition of orbifold cohomology that will be used in this paper. Let $\cX$ be a smooth finite type equidimensional tame Artin stack over $\C$ with affine diagonal.

\begin{notation}
Let $\ell \in \Z_{>0}$, let $k'$ be a field extension of $k$, and let $f: B\mu_\ell \otimes_{k} k' \to \cX$ be a morphism. The pullback $f^*T_\cX$ of the tangent bundle of $\cX$ is a rank $\dim\cX$ vector bundle on $B\mu_\ell \otimes_k k'$ and is therefore a $\dim\cX$-dimensional vector space over $k'$ equipped with a $\mu_\ell$-action. There exist $a_1, \dots, a_{\dim\cX} \in \{1, \dots, \ell\}$ and a basis $v_1, \dots, v_{\dim\cX}$ for $f^*T_\cX$ such that for all $i \in \{1, \dots, \dim\cX\}$, each $\xi \in \mu_\ell$ sends $v_i$ to $\xi^{-a_i} v_i$. Therefore,
\[
	f^*T_\cX \cong \bigoplus_{i=1}^{\dim\cX} \cO_{B\mu_{\ell} \otimes_k k'}(a_i).
\]
We set
\[
	\shft_\cX(f) = \dim\cX - (1/\ell)\sum_{i =1}^{\dim\cX}a_i.
\]
It is well known that if $k''$ is a field extension of $k$ and $g: B\mu_\ell \otimes_{k} k'' \to \cX$ is a morphism in the same connected component of the cyclotomic inertia stack $I_{\mu_\ell}(\cX)$ as $f$, then $\shft_\cX(f) = \shft_\cX(g)$. See e.g., \cite[Definition 3.8]{Yasuda2004} or \cite[Proof of Proposition 4.19]{SatrianoUsatine6}. Thus if $\cY$ is a connected component of $I_{\mu_\ell}(\cX)$, we set
\[
	\shft_\cX(\cY) = \shft_\cX(f),
\]
where $f: B\mu_\ell \otimes_{k} k' \to \cX$ is a point of $\cY$.
\end{notation}

\begin{remark}\label{YasudaShiftConvention}
We note that in \cite{Yasuda2004, Yasuda}, the action that Yasuda considered on the tangent space is the inverse of the action we consider above (see e.g., \cite[Proof of Lemma 68]{Yasuda}), so the number $\shft_\cX(\cY)$ coincides with the shift defined in \cite{Yasuda2004, Yasuda}.
\end{remark}

\begin{definition}
For each $i \in \Q$, set
\[
	H^{\even, i}_{\orb}(\cX) = \bigoplus_{\cY} H_c^{i-2\shft_\cX(\cY)}(\cY)(-\shft_\cX(\cY)),
\]
where $\cY$ varies over connected components of cyclotomic inertia $\bigsqcup_{\ell \in \Z_{>0}} I_{\mu_\ell}(\cX)$ of $\cX$ such that $i-\shft_\cX(\cY)$ is an even integer, set
\[
	H^{\odd, i}_{\orb}(\cX) = \bigoplus_{\cY} H_c^{i-2\shft_\cX(\cY)}(\cY)(-\shft_\cX(\cY)),
\]
where $\cY$ varies over connected components of cyclotomic inertia $\bigsqcup_{\ell \in \Z_{>0}} I_{\mu_\ell}(\cX)$ of $\cX$ such that $i-\shft_\cX(\cY)$ is an odd integer, and set
\[
	H^i_{\orb}(\cX) = H^{\even, i}_{\orb}(\cX) \oplus H^{\odd, i}_{\orb}(\cX).
\]
Furthermore, set
\[
	\HD_{\orb}(\cX) = \sum_{i \in \Q} \left(\HD(H_{\orb}^{\even, i}(\cX)) -  \HD(H_{\orb}^{\odd, i}(\cX))\right).
\]
\end{definition}

\begin{remark}
If in addition we assume $\cX$ is proper, then each $\cY$ above is proper, so $H^*_c(\cY) \cong H^*(\cY)$, and since we are taking cohomology with rational coefficients, $H^*(\cY) \cong H^*(Y)$, where $Y$ is the coarse moduli space of $\cY$. Thus in this case, our definition of $H^*_{\orb}(\cX)$ coincides with the algebraic version, e.g., \cite[Definition 75]{Yasuda}, of Chen and Ruan's orbifold cohomology \cite{ChenRuan}. More precisely in this case, $H^{\even, *}_{\orb}(\cX)$ and $H^{\odd, *}_{\orb}(\cX)$ coincide with the notions (to our knowledge first introduced) in \cite[Proof of Corollary 77]{Yasuda}, which provide useful bookkeeping for dealing with sign when working with $\cX$ that have nonvanishing orbifold cohomology in fractional degree. This is where we obtained the idea to similarly define $H^{\even, *}_{\str}(\cX)$ and $H^{\odd, *}_{\str}(\cX)$ separately.
\end{remark}

\section{The stringy class of a stack}

The goal of this section is to prove the following theorem.

\begin{theorem}\label{theoremStringyClassWellDefined}
Let $\cX$ be a smooth finite type equidimensional Artin stack over $k$ with affine diagonal and properly stable good moduli space, let $w \in \Q$, and set $\cA = \cA_\cX \cap \wt_\cX^{-1}(w)$.

\begin{enumerate}[(a)]

\item\label{CanonicalSetOfTwistedArcsIsMeasurable} The subset $\cA \subset |\sJ(\cX)|$ is measurable.

\item\label{CanonicalSetOfTwistedArcsHasConstructibleImage} For each $n \in \Z_{\geq 0}$, we have $\theta_n(\cA)$ is a quasi-compact locally constructible subset of $|\sJ_n(\cX)|$. In particular, $\theta_n^{-1}(\theta_n(\cA)) \subset |\sJ(\cX)|$ is a bounded cylinder.

\item We have the equality
\[
	\nu_\cX(\cA) = \lim_{n \to \infty}(\nu_\cX(\theta_n^{-1}(\theta_n(\cA))).
\]
	
\end{enumerate}
\end{theorem}

Before we prove \autoref{theoremStringyClassWellDefined}, we use it to prove the following corollary.

\begin{corollary}\label{weightFunctionIsIntegrable}
Let $\cX$ be a smooth finite type equidimensional Artin stack over $k$ with affine diagonal and properly stable good moduli space. Then $\bL^{-\wt_\cX}$ is integrable on $\cA_\cX$.
\end{corollary}

\begin{proof}
This is an immediate consequence of \autoref{theoremStringyClassWellDefined}(\ref{CanonicalSetOfTwistedArcsIsMeasurable}) and the fact that $\wt_\cX$ takes only finitely many values on $\cA_\cX$. The latter is a direct consequence of \cite[Proposition 4.19]{SatrianoUsatine6} and the fact that if $\cX' \to \cX$ is the canonical reduction of stabilizers of Edidin--Rydh, then $\cX'$ is tame \cite[Theorem 2.11(4c)]{EdidinRydh}.
\end{proof}

Now we return to proving \autoref{theoremStringyClassWellDefined}. We will begin with a few lemmas.

\begin{lemma}\label{surjectiveRepresentableMorphismBoundDimension}
Let $\cY \to \cZ$ be a representable morphism of finite type Artin stacks over $k$. If the induced map $|\cY| \to |\cZ|$ is surjective, then $\dim\cZ \leq \dim\cY$.
\end{lemma}

\begin{proof}
Let $\cW$ be an integral closed substack of $\cZ$ with $\dim\cW = \dim\cZ$, let $\cU$ be the smooth locus of $\cW$, and set $\cV = \cU \times_\cZ \cY$. It is sufficient to show that $\dim\cV \geq \dim \cU$. Let $U \to \cU$ be a smooth cover of constant relative dimension by a finite type scheme, and set $V = \cV \times_\cU U$. Then it is sufficient to show that $\dim V \geq \dim U$. Since $V$ is a finite type algebraic space over $k$, there exists an \'etale cover $V' \to V$ by a finite type scheme over $k$, and $\dim V' = \dim V$. Thus it is sufficient to show that $\dim V' \geq \dim U$, but this follows from the fact that $V' \to U$ is a surjective map of finite type schemes over $k$.
\end{proof}

\begin{lemma}\label{containedConstructiblesBound}
Let $\cY$ be a finite type Artin stack over $k$ with affine geometric stabilizers, and let $\cC \subset \cD$ be constructible subsets of $|\cY|$. Then
\[
	\Vert \e(\cC) \Vert \leq \Vert \e(\cD) \Vert.
\]
\end{lemma}

\begin{proof}
By considering partitions of $\cC$ and $\cD \setminus \cC$ into locally closed substacks of $\cY$, this is a straightforward consequence of \autoref{RecallingNormAndDimension}.
\end{proof}

\begin{lemma}\label{representableConstructibleBound}
Let $\cY$ and $\cZ$ be finite type Artin stacks over $k$ with affine geometric stabilizers, let $\cC$ and $\cD$ be constructible subsets of $|\cY|$ and $|\cZ|$, respectively, and let $\pi: \cY \to \cZ$ be a representable morphism such that $\cD \subset \pi(\cC)$. Then
\[
	\Vert \e(\cD) \Vert \leq \Vert \e(\cC) \Vert.
\]
\end{lemma}

\begin{proof}
By Chevalley's Theorem for Artin stacks \cite[Theorem 5.1]{HallRydh}, $\pi(\cC)$ is a constructible subset of $|\cZ|$. Thus by \autoref{containedConstructiblesBound} it is sufficient to prove the special case where $\cD = \pi(\cC)$.

 Let $\{\cZ_i\}_i$ be a partition of $\cD$ by locally closed substacks of $\cZ$, and for each $i$, let $\{\cY_{i,j}\}_j$ be a partition of $\cC \cap \pi^{-1}(|\cZ_i|)$ by reduced locally closed substacks of $\cY$. Set $\cY' = \bigsqcup_{i,j} \cY_{i,j}$ and $\cZ' = \bigsqcup_i \cZ_i$. Then the induced map $\cY' \to \cZ'$ induces a surjection on associated topological spaces, so by \autoref{surjectiveRepresentableMorphismBoundDimension},
\[
	\dim\cY' \geq \dim\cZ'.
\]
Therefore
\[
	\Vert \e(\cD) \Vert = \Vert \e(\cZ') \Vert = \exp(\dim\cZ') \leq \exp(\dim\cY') = \Vert \e(\cY') \Vert = \Vert \e(\cC) \Vert,
\]
where the second and third equalities are given by \autoref{RecallingNormAndDimension}.
\end{proof}

\begin{lemma}\label{twistedJetFunctorSendsRepresentableToRepresentable}
Let $\cX'$ and $\cX$ be finite type Artin stacks over $k$ with affine diagonal, and let $\cX' \to \cX$ be a representable map. Then for each $n \in \Z_{\geq 0}$, the induced map $\sJ_n(\cX') \to \sJ_n(\cX)$ is representable.
\end{lemma}

\begin{proof}
Let $g$ be the map $\cX' \to \cX$. Let $\ell \in \Z_{>0}$, let $A$ be a $k$-algebra, let $f'_i\colon\cD^\ell_n \otimes_k A \to\cX'$ for $i=1,2$ be representable maps, and let $\alpha\colon f'_1\Rightarrow f'_2$ be a $2$-isomorphism. Then we have an induced $2$-isomorphism $g^*(\alpha)\colon gf'_1\Rightarrow gf'_2$. If $g^*(\alpha)=\id$ then by representability of $g$, we see $\alpha=\id$. Thus, $\cJ_n(\cX')\to\cJ_n(\cX)$ is representable.
\end{proof}

For the remainder of this section, let $\cX$ be a smooth finite type equidimensional Artin stack over $k$ with affine diagonal and properly stable good moduli space, let $\pi: \cX' \to \cX$ be the canonical reduction of stabilizers of Edidin--Rydh, let $\cX^s$ be the stable locus of $\cX$, and let $\cU' = \pi^{-1}(\cX^s)$.

\begin{proposition}\label{imageOfBoundedCylinderAndInequality}
If $\cC' \subset |\sJ(\cX')|$ is a bounded cylinder disjoint from $|\sJ(\cX' \setminus \cU')|$ and $\cC$ is the image of $\cC'$ in $|\sJ(\cX)|$, then $\cC$ is a bounded cylinder and 
\[
	\Vert \nu_{\cX'}(\cC') \Vert \geq \Vert \nu_{\cX}(\cC) \Vert.
\]
\end{proposition}

\begin{proof}
We have that $\cC$ is a bounded cylinder by \cite[Lemma 12.5]{SatrianoUsatine6}. Since $\cX' \to \cX$ is representable, $\sJ_n(\cX') \to \sJ_n(\cX)$ is representable for all $n \in \Z_{\geq 0}$ by \autoref{twistedJetFunctorSendsRepresentableToRepresentable}. Thus \autoref{representableConstructibleBound} implies that $\Vert \e(\theta_n(\cC')) \Vert \geq \Vert \e(\theta_n(\cC)) \Vert$ for all $n \in \Z_{\geq 0}$, so $\Vert \nu_{\cX'}(\cC') \Vert \geq \Vert \nu_{\cX}(\cC) \Vert$.
\end{proof}

We are now prepared for the proof of \autoref{theoremStringyClassWellDefined}.

\begin{proof}[Proof of \autoref{theoremStringyClassWellDefined}]
Let $\cY$ be a closed substack of $\cX$ supported on $|\cX| \setminus |\cX^s|$, and set $\cY' = \cY \times_\cX \cX'$. Let $\cC'$ be the preimage of $\wt_\cX^{-1}(w)$ in $|\sJ(\cX')|$. By \cite[Remark 4.21]{SatrianoUsatine6}, $\cC'$ is a cylinder. Furthermore $\cC'$ is a bounded cylinder since $\cX'$ is finite type and tame. For each $n \in \Z_{\geq 0}$, set $\cC'_n = \cC' \cap \theta_n^{-1}(|\sJ_n(\cX')| \setminus |\sJ_n(\cY')|)$. Then
\begin{itemize}

\item each $\cC'_n$ is a bounded cylinder disjoint from $|\sJ(\cX' \setminus \cU')|$,

\item $\cC'_n \subset \cC'_{n+1}$ for all $n$, and

\item $\cC' \setminus |\sJ(\cX' \setminus \cU')|$ is the union of all $\cC'_n$.

\end{itemize}
Since $|\sJ(\cX' \setminus \cU')|$ is measurable by \cite[Theorem 9.1]{SatrianoUsatine6}, the union of all the $\cC'_n$ is measurable. Therefore
\[
	\lim_{n \to \infty} \nu_{\cX'}(\cC'_{n+1} \setminus \cC'_n) = 0.
\]
For each $n$, let $\cC_n$ be the image of $\cC'_n$ in $|\sJ(\cX)|$. Then by \autoref{imageOfBoundedCylinderAndInequality}, each $\cC_n$ is a bounded cylinder. Also each $\cC_{n+1} \setminus \cC_n$ is contained in the image of $\cC'_{n+1} \setminus \cC'_n$, so by \autoref{containedConstructiblesBound} and \autoref{imageOfBoundedCylinderAndInequality},
\[
	\Vert \nu_{\cX}(\cC_{n+1} \setminus \cC_n) \Vert \leq \Vert \nu_{\cX'}( \cC'_{n+1} \setminus \cC'_n) \Vert.
\]
Therefore
\[
	\lim_{n \to \infty} \nu_{\cX}(\cC_{n+1} \setminus \cC_n) = 0,
\]
so $\bigcup_{n \to \infty} \cC_n$ is measurable and has measure $\lim_{n \to \infty} \nu_\cX(\cC_n)$. In what follows, note that for each $\ell \in \Z_{\geq 0}$, we have $|\sJ^\ell(\cY)|$ is measurable and $\nu_\cX(|\sJ^\ell(\cY)|) = 0$ by \cite[Theorem 9.1]{SatrianoUsatine6}. Since $\cX'$ is finite type and tame, there exists some $r \in \Z_{\geq 0}$ such that $\sJ^\ell(\cX') = \emptyset$ for all $\ell > r$.

We will now prove parts (a), (b), and (c).
\begin{enumerate}[(a)]

\item We have $\cA \setminus |\bigcup_{\ell = 1}^r \sJ^\ell(\cY)| = \bigcup_{n \to \infty} \cC_n$ is measurable, so $\cA$ is measurable.

\item This part follows from Chevalley's Theorem for Artin stacks, the fact that $\theta_n(\cA)$ is the image of $\theta_n(\cC')$, and the fact that $\cC'$ is bounded.

\item Note that we have already shown that
\[
	\nu_\cX(\cA) = \nu_\cX \left(\cA \setminus |\bigcup_{\ell = 1}^r \sJ^\ell(\cY)|\right) = \lim_{n \to \infty} \nu_\cX(\cC_n).
\]
For each $n$, set $\cE_n = \theta_n^{-1}( | \bigcup_{\ell = 1}^r \sJ^\ell_n(\cY) | )$. Then each $\cE_n \supset \cE_{n+1}$ and $\bigcap_{n = 0}^\infty \cE_n = \bigcup_{\ell = 1}^r |\sJ^\ell(\cY)|$ has measure 0, so $\lim_{n \to 0} \nu_\cX(\cE_n) = 0$. Also for all n,
\[
	\cA \setminus \cC_n \subset \cE_n.
\]
Thus for all $m, n$,
\[
	\theta_n(\cA) \setminus \theta_n(\cC_m) \subset \theta_n(\cA \setminus \cC_m) \subset \theta_n(\cE_m),
\]
so for all $m \leq n$,
\begin{align*}
	\Vert \e(\theta_n(\cA))\bL^{-(n+1)\dim\cX} -& \e(\theta_n(\cC_m))\bL^{-(n+1)\dim\cX} \Vert \\
	&\leq \Vert \e(\theta_n(\cE_m))\bL^{-(n+1)\dim\cX} \Vert = \Vert \nu_\cX(\cE_m) \Vert.
\end{align*}
Let $\varepsilon > 0$, and let $m$ be such that $\Vert \nu_\cX(\cE_m) \Vert < \varepsilon$ and $\Vert \nu_\cX(\cA) - \nu_\cX(\cC_m) \Vert < \varepsilon$. Since $\cC_m$ is a bounded cylinder, there exists some $m'$ such that for all $n \geq m'$ we have $\nu_\cX(\cC_m) = \e(\theta_n(\cC_m))\bL^{-(n+1)\dim\cX}$. Then for all $n \geq \max(m, m')$,
\begin{align*}
	\Vert \nu_\cX(\theta_n^{-1}(&\theta_n(\cA))) - \nu_\cX(\cC_m) \Vert \\
	&= \Vert \e(\theta_n(\cA))\bL^{-(n+1)\dim\cX} - \e(\theta_n(\cC_m))\bL^{-(n+1)\dim\cX} \Vert \\
	&\leq \Vert \nu_\cX(\cE_m) \Vert \\
	&< \varepsilon,
\end{align*}
so
\begin{align*}
	\Vert \nu_\cX(\cA) &- \nu_\cX(\theta_n^{-1}(\theta_n(\cA))) \Vert \\
	&\leq \max\left(\Vert \nu_\cX(\cA) - \nu_\cX(\cC_m) \Vert, \Vert \nu_\cX(\theta_n^{-1}(\theta_n(\cA))) - \nu_\cX(\cC_m) \Vert \right) \\
	&< \varepsilon.
\end{align*}
Therefore
\[
	\nu_\cX(\cA) = \lim_{n \to \infty}(\nu_\cX(\theta_n^{-1}(\theta_n(\cA))).
\]

\end{enumerate}
\end{proof}

\section{A criterion for integrality of the weight function}

Throughout this section, let $\cX$ be a smooth finite type equidimensional Artin stack over $\C$ with affine diagonal and properly stable good moduli space. Then by \autoref{weightFunctionIsIntegrable}, $\bL^{-\wt_\cX}$ is integrable on $\cA_\cX$, so we may consider
\[
	\HD_{\str}(\cX) = \HD\left(\bL^{\dim\cX} \int_{\cA_\cX} \bL^{-\wt_\cX} \diff\nu_\cX \right) \in \Z \lPuiseux u^{-1}, v^{-1} \rPuiseux.
\]
If $\wt_\cX(\cA_\cX) \subset \Z$, then 
\[
	\int_{\cA_\cX} \bL^{-\wt_\cX} \diff\nu_\cX \in \widehat{\sM_k},
\]
so in that case,
\[
	\HD_{\str}(\cX) \in \Z \llparenthesis u^{-1}, v^{-1} \rrparenthesis.
\]
In this section we will prove the following converse.

\begin{proposition}\label{propositionIntegralWeightFunction}
If $\HD_{\str}(\cX) \in \Z \llparenthesis u^{-1}, v^{-1} \rrparenthesis$, then $\wt_\cX(\cA_\cX) \subset \Z$.
\end{proposition}

We begin with some lemmas.

\begin{lemma}\label{lemmaDiagonalTermsOfHDPolyOfStack}
Let $\cY$ be a finite type Artin stack over $\C$ with affine geometric stabilizers.

\begin{enumerate}

\item If $p \in \Z$ such that $p > \dim\cY$, then the coefficient of $u^p v^p$ in $\HD(\e(\cY))$ is zero.

\item The coefficient of $u^{\dim\cY} v^{\dim\cY}$ in $\HD(\e(\cY))$ is positive.

\end{enumerate}

\end{lemma}

\begin{proof}
Noting that $\e(\GL_{n,k})$ is a polynomial in $\bL$ and therefore $\HD(\e(\GL_{n,k}))$ is a polynomial in $uv$, the exact same proof as in \cite[Lemma 4.2]{SatrianoUsatine5} gives the desired result.
\end{proof}

\begin{lemma}\label{lemmaHDVolumesDiagonalTerms}
Let $w \in \Q$, set $\cA = \cA_\cX \cap \wt_\cX^{-1}(w)$, and set $d = \log \Vert \nu_\cX(\cA) \Vert \in \Z \cup \{-\infty\}$.

\begin{enumerate}

\item If $p \in \Z$ such that $p > d$, then the coefficient of $u^p v^p$ in $\HD(\nu_\cX(\cA))$ is zero.

\item If $\cA \neq \emptyset$, then $\nu_\cX(\cA) \neq 0$ and the coefficient of $u^d v^d$ in $\HD(\nu_\cX(\cA))$ is positive.

\end{enumerate}
\end{lemma}

\begin{proof}
If $\cA = \emptyset$, there is nothing to show, so we will assume $\cA \neq \emptyset$. 

We will first show that $\nu_\cX(\cA) \neq 0$. Let $\pi: \cX' \to \cX$ be the canonical reduction of stabilizers of Edidin--Rydh, let $\cX^s$ be the stable locus of $\cX$, let $\cY'$ be a closed substack of $\cX'$ supported on $|\cX'| \setminus |\pi^{-1}(\cX^s)|$, let $\cC'$ be the preimage of $\wt_{\cX}^{-1}(w)$ in $|\sJ(\cX')|$, and for each $n \in \Z_{\geq 0}$, set
\[
	\cC'_n = \cC' \cap \theta_n^{-1}(|\sJ_n(\cX')| \setminus |\sJ_n(\cY')|),
\]
and let $\cC_n$ be the image of $\cC'_n$ in $|\sJ(\cX)|$. By \cite[Remark 4.21]{SatrianoUsatine6}, $\cC'$ is a cylinder, and it is bounded since $\cX'$ is finite type and tame. Thus each $\cC'_n$ is a bounded cylinder disjoint from $|\sJ(\cY')|$, so by \autoref{imageOfBoundedCylinderAndInequality}, each $\cC_n$ is a bounded cylinder. By construction, each $\cC_n \subset \cA$. Now since $\sJ(\pi)(\cC') = \cA \neq \emptyset$, we have $\cC' \neq \emptyset$. Since a nonempty bounded cylinder has nonzero volume, e.g., by \autoref{RecallingNormAndDimension}, we have $\nu_{\cX'}(\cC') \neq 0$. Thus since $\nu_{\cX'}(|\sJ(\cY')|) = 0$ by \cite[Theorem 9.1]{SatrianoUsatine6}, we have $\cC' \not\subset |\sJ(\cY')|$. Therefore since $|\sJ(\cY')| = \bigcap_{n \in \Z_{\geq 0}} \theta^{-1}_n(|\sJ_n(\cY')|)$, there exists some $m \in \Z_{\geq 0}$ such that $\cC' \not\subset \theta^{-1}_m(|\sJ_m(\cY')|)$. Then $\cC'_m \neq \emptyset$, so $\cC_m \neq \emptyset$. Using again that a nonempty bounded cylinder has nonzero volume, we have that $\nu_\cX(\cC_m) \neq 0$. Then since $\cC_m \subset \cA$, we have that $\nu_\cX(\cA) \neq 0$.

For each $n \in \Z_{\geq 0}$, set $\cA_n = \theta_n^{-1}(\theta_n(\cA))$. Then by \autoref{theoremStringyClassWellDefined}, each $\cA_n$ is a bounded cylinder, and
\[
	\nu_\cX(\cA) = \lim_{n \to \infty}(\nu_\cX(\cA_n)).
\]
Thus there exists some $r \in \Z_{\geq 0}$ such that
\[
	\log \Vert \nu_\cX(\cA) - \nu_\cX(\cA_r) \Vert < d.
\]
Therefore every nonzero monomial of $\HD(\nu_\cX(\cA) - \nu_\cX(\cA_r))$ has total degree strictly less than $2d$. Thus it is sufficient to show that for $p > d$, the coefficient of $u^pv^p$ in $\HD(\nu_\cX(\cA_r))$ is zero and that the coefficient of $u^dv^d$ in $\HD(\nu_\cX(\cA_r))$ is positive.

Since $\cA_r$ is a bounded cylinder, there exists a finite type Artin stack $\cZ$ over $\C$ with affine geometric stabilizers such that $\nu_\cX(\cA_r) = \e(\cZ)$. Since \[
	\Vert \nu_\cX(\cA) - \nu_\cX(\cA_r) \Vert < \exp(d) = \Vert \nu_\cX(\cA) \Vert,
\]
the non-Archimedian triangle inequality gives that 
\[
	\Vert \nu_\cX(\cA_r) \Vert = \Vert \nu_\cX(\cA) \Vert = \exp(d).
\]
Therefore $\dim\cZ = d$ by \autoref{RecallingNormAndDimension}. Thus the desired statement about the coefficients of $\HD(\nu_\cX(\cA_r)) = \HD(\e(\cZ))$ follows from \autoref{lemmaDiagonalTermsOfHDPolyOfStack}, and we are done.
\end{proof}

Before we state the next proposition, recall that the function $\wt_\cX$ takes only finitely many values on $\cA_\cX$ by \cite[Proposition 4.19]{SatrianoUsatine6} and the fact that, since the canonical reduction of stabilizers of $\cX$ is tame, we have $\cA_\cX$ is bounded. Note also that for any $w \in \wt_\cX(\cA_\cX)$, we have $\log \Vert \nu_\cX(\cA_\cX \cap \wt_\cX^{-1}(w)) \Vert \in \Z$ by \autoref{lemmaHDVolumesDiagonalTerms}.

\begin{proposition}\label{propositionPositiveCoefficientForNonIntegerTerm}
Suppose $\wt_\cX(\cA_\cX) \not\subset \Z$, and let $q$ be an element of $\wt_\cX(\cA_\cX) \setminus \Z$ that maximizes 
\[
	-q + \log \Vert \nu_\cX(\cA_\cX \cap \wt_\cX^{-1}(q)) \Vert \in \Q.
\]
Then the coefficient of $(uv)^{\dim\cX - q + \log \Vert \nu_\cX(\cA_\cX \cap \wt_\cX^{-1}(q)) \Vert}$ in $\HD_\str(\cX)$ is positive.
\end{proposition}

\begin{proof}
For each $w \in \Q$, set $\cA^w = \cA_\cX \cap \wt_{\cX}^{-1}(w)$. Then 
\[
	\HD_{\str}(\cX) = \HD\left(\sum_{w \in \Q} \bL^{\dim\cX-w} \nu_\cX(\cA^w) \right).
\]
Now set
\[
	f(u,v) = \HD\left(\sum_{w \in \Q \setminus \Z} \bL^{\dim\cX-w} \nu_\cX(\cA^w) \right).
\]
Then
\[
	\HD_{\str}(\cX) - f(u,v) = \HD\left(\sum_{w \in \Z} \bL^{\dim\cX-w} \nu_\cX(\cA^w) \right) \in \Z \llparenthesis u^{-1}, v^{-1} \rrparenthesis.
\]
Therefore since $\dim\cX - q + \log \Vert \nu_\cX(\cA^q) \Vert \notin \Z$, it is sufficient to show that the coefficient of $(uv)^{\dim\cX - q + \log \Vert \nu_\cX(\cA^q) \Vert}$ in $f(u,v)$ is positive, but this follows from our choice of $q$ and \autoref{lemmaHDVolumesDiagonalTerms}.
\end{proof}

\begin{proof}[Proof of \autoref{propositionIntegralWeightFunction}]
The desired result follows immediately from \autoref{propositionPositiveCoefficientForNonIntegerTerm} and the fact that in its statement, $\dim\cX - q + \log \Vert \nu_\cX(\cA_\cX \cap \wt_\cX^{-1}(q))\Vert \notin \Z$.
\end{proof}

\section{Smoothness of the truncation morphisms}\label{sectionSmoothnessOfTruncationMorphisms}

The goal of this section is to prove the following theorem.

\begin{theorem}\label{theoremSmoothnessOfTruncation}
Let $\cX$ be a smooth equidimensional finite type Artin stack over $k$ with affine diagonal and a good moduli space, and let $n \in \Z_{\geq 0}$. Then $\sJ_n(\cX)$ is equidimensional, smooth over $k$, and has dimension $(n+1)\dim\cX$, and the truncation morphism $\theta^{n+1}_n: \sJ_{n+1}(\cX) \to \sJ_n(\cX)$ is smooth of constant relative dimension $\dim\cX$.
\end{theorem}

Throughout this section if $n \in \Z_{\geq 0}$, $\ell \in \Z_{>0}$, and $A$ is a $k$-algebra, we will let $\cD^\ell_{n, A}$ denote the stack quotient $[\Spec( A[t^{1/\ell}]/(t^{n+1}) ) / \mu_\ell]$, where $\xi \in \mu_\ell$ acts on $A[t^{1/\ell}]/(t^{n+1})$ by $f(t^{1/\ell}) \mapsto f(\xi t^{1/\ell})$. In the special case $\ell =1$, we will use the notation $D_{n, A} = \cD^1_{n,A} = \Spec(A[t]/(t^{n+1}) ).$

We will need the following lemma.

\begin{lemma}\label{truncatedTwistedDiscRankDimensionGlobalSections}
Let $n \in \Z_{\geq 0}$, let $\ell \in \Z_{>0}$, and let $k'$ be a field extension of $k$. If $\cF$ is a rank $r$ locally free sheaf on $\cD^\ell_{n, k'}$, then $\dim_{k'}H^0(\cF) = (n+1)r$.
\end{lemma}

\begin{proof}
By \cite[Proposition 3.1]{SatrianoUsatine6}, $\cF$ is a direct sum of coherent sheaves isomorphic to $\cO_{\cD^\ell_{n, k'}} (a)$ for some $a \in \{0, \dots, \ell-1\}$. Thus we only need to show that $\dim_{k'}H^0( \cO_{\cD^\ell_{n, k'}} (a) ) = n+1$, but this follows from the fact that the $k'$ vector space of $\mu_\ell$-invariants of $(k'[ t^{1/\ell} ] / (t^{n+1}))(a)$ is spanned by $t^{a/\ell}, t^{1 + a/\ell}, \dots, t^{n+a/\ell}$.
\end{proof}

If $\cY$ is an Artin stack over $k$, it is natural to think of the fibers of the truncation map $\sL_1(\cY) \to \sL_0(\cY) = \cY$ as a stacky version of tangent spaces of $\cY$. The next proposition describes these stacky tangent spaces for $\sJ_n(\cX)$ and will later be used to compute the dimension of $\sJ_n(\cX)$.

\begin{proposition}\label{computingStackyTangentSpaceOfTwistedJetStack}
Let $\cX$ be a smooth equidimensional finite type Artin stack over $k$ with affine diagonal and a good moduli space, let $n \in \Z_{\geq 0}$, let $k'$ be a field extension of $k$, and let $\Spec(k') \to \sJ_n(\cX)$ be a $k$-morphism. Then there exist $r, s \in \Z_{\geq 0}$ such that $r-s = (n+1)\dim\cX$ and
\[
	\sL_1(\sJ_n(\cX)) \times_{\sJ_n(\cX)} \Spec(k') \cong \bA^{r}_{k'} \times_{k'} B\bG_{a,k'}^{s}
\]
as stacks over $k'$. In particular,
\[
	\dim\left( \sL_1(\sJ_n(\cX)) \times_{\sJ_n(\cX)} \Spec(k')\right) = (n+1)\dim\cX.
\]
\end{proposition}

\begin{proof}
Set $\cY = \sL_1(\sJ_n(\cX)) \times_{\sJ_n(\cX)} \Spec(k')$. Let $\varphi: \cD^\ell_{n, k'} \to \cX$ be the representable map corresponding to $\Spec(k') \to \sJ_n(\cX)$, and consider the $k'$ vector spaces
\[
	E^0 = \Ext^0_{\cD^\ell_{n, k'}}(L\varphi^* L_\cX, \cO_{\cD^\ell_{n, k'}})
\]
and
\[
	E^{-1} = \Ext^{-1}_{\cD^\ell_{n, k'}}(L\varphi^* L_\cX, \cO_{\cD^\ell_{n, k'}}).
\]
Set $r = \dim_{k'} E^0$ and $s = \dim_{k'} E^{-1}$. If $A$ is a $k'$-algebra, let $\varphi_A$ be the composition $\cD^\ell_{n, A} \to \cD^\ell_{n, k'} \xrightarrow{\varphi} \cX$. Since $\Spec(A) \to \Spec(A[t]/(t^2))$ has the section given by the inclusion $A \to A[t]/(t^2)$, we have that $\cY(A)$ is nonempty. Therefore deformation theory gives that
\[
	\overline{\cY}(A) = \Ext^0_{\cD^\ell_{n, A}}(L\varphi_A^* L_\cX, \cO_{\cD^\ell_{n, A}}) = E^0 \otimes_{k'} A
\]
and that each object of $\cY(A)$ has automorphism group
\[
	\Ext^{-1}_{\cD^\ell_{n, A}}(L\varphi_A^* L_\cX, \cO_{\cD^\ell_{n, A}}) = E^{-1} \otimes_{k'} A,
\]
so
\[
	\cY \cong \bA_{k'}^{r} \times_{k'} B\bG_{a,k'}^{s}.
\]
We now note that the statement and proof of \cite[Lemma 2.2]{SatrianoUsatine6} remain true if we replace each instance of $k'\llbracket t \rrbracket$ with $k'[t]/(t^{n+1})$, and therefore there exist $r', s' \in \Z_{\geq 0}$ with $r'-s' = \dim\cX$ and an exact triangle
\[
	L\varphi^*L_\cX \to \cF \to \cG
\]
where $\cF$ and $\cG$ are locally free sheaves of ranks $r'$ and $s'$, respectively. By \autoref{truncatedTwistedDiscRankDimensionGlobalSections}, 
\[
	\dim_{k'}H^0\left(\mathscr{H}om(\cF, \cO_{\cD^\ell_{n, k'}})\right) = (n+1)r',
\]
and
\[
	\dim_{k'}H^0\left(\mathscr{H}om(\cG, \cO_{\cD^\ell_{n, k'}})\right) = (n+1)s'.
\]
Therefore by the exact sequence
\[
	0 \to E^{-1} \to H^0\left(\mathscr{H}om(\cG, \cO_{\cD^\ell_{n, k'}})\right) \to H^0\left(\mathscr{H}om(\cF, \cO_{\cD^\ell_{n, k'}})\right)  \to E^0 \to 0,
\]
we have
\[
	r - s = (n+1)r' - (n+1)s' = (n+1)\dim\cX,
\]
and we are done.
\end{proof}

We are now prepared to prove the first part of \autoref{theoremSmoothnessOfTruncation}.

\begin{proposition}\label{TwistedJetStacksAreSmoothAndEquidimensionalOfCorrectDimension}
Let $\cX$ be a smooth equidimensional finite type Artin stack over $k$ with affine diagonal and a good moduli space, and let $n \in \Z_{\geq 0}$. Then $\sJ_n(\cX)$ is smooth over $k$ and equidimension $(n+1)\dim\cX$.
\end{proposition}

\begin{proof}
By \cite[Remark 3.13(v)]{Rydh}, the stack $\uHom_k(\cD^\ell_n, \cX)$ is smooth over $k$. Since $\sJ^\ell_n(\cX)$ is an open substack of $\uHom_k(\cD^\ell_n, \cX)$ by \cite[Proposition 4.2]{SatrianoUsatine6}, we have that $\sJ_n(\cX)$ is smooth over $k$. Now let $\ell \in \Z_{>0}$ and $\cY$ be a connected component of $\sJ^\ell_n(\cX)$. Then by \cite[Lemma 3.34]{SatrianoUsatine1}, every fiber of $\sL_1(\cY) \to \cY$ has dimension equal to $\dim\cY$. On the other hand by \autoref{computingStackyTangentSpaceOfTwistedJetStack}, these fibers all have dimension $(n+1)\dim\cX$, and we are done.
\end{proof}

In order to finish the proof of \autoref{theoremSmoothnessOfTruncation}, we will need the following version of miracle flatness. It should come as no surprise that miracle flatness holds as follows for Artin stacks, but nonetheless, we include a proof for sake of completeness.

\begin{proposition}[Miracle flatness for Artin stacks]\label{miracleFlatnessForStacks}
Let $\cY$ and $\cZ$ be equidimensional finite type Artin stacks over $k$, and let $\cY \to \cZ$ be a morphism. Assume that $\cY$ is Cohen--Macaulay and $\cZ$ is regular. If for every field extension $k'$ of $k$ and every morphism $\Spec(k') \to \cZ$ we have that $\dim(\cY \times_\cZ \Spec(k')) = \dim\cY - \dim\cZ$, then $\cY \to \cZ$ is flat.
\end{proposition}

\begin{remark}
We call a stack \emph{Cohen--Macaulay} (resp. \emph{regular}) if it has a smooth cover by a scheme that is Cohen--Maucalay (resp. regular). Since being Cohen--Macaulay (resp. regular) is local in the smooth topology \cite[Tag 036A (resp. Tag 036D)]{stacks-project}, this is equivalent to all smooth covers by a scheme being Cohen--Macaulay (resp. regular).
\end{remark}

\begin{proof}
Let $Z \to \cZ$ be a finite type smooth cover of constant relative dimension with $Z$ a scheme, let $\cW = Z \times_\cZ \cY$, and let $W \to \cW$ be a finite type smooth cover of constant relative dimension with $W$ a scheme. Then $W$ and $Z$ are equidimensional finite type schemes over $k$, $W$ is Cohen--Macaulay, and $Z$ is regular. Furthermore, if $k'$ is a field extension of $k$ and $\Spec(k') \to Z$ is a morphism, then $W \times_Z \Spec(k')$ is a smooth cover of $\cW \times_Z \Spec(k') \cong \cY \times_\cZ \Spec(k')$ with the map $W \times_Z \Spec(k') \to \cW \times_Z \Spec(k')$ having constant relative dimension $\dim W - \dim\cW$. Thus
\begin{align*}
	\dim(W \times_Z \Spec(k')) &= \dim(\cY \times_\cZ \Spec(k')) + \dim W - \dim\cW\\
	&= \dim \cY - \dim \cZ + \dim W - \dim\cW \\
	&= \dim \cW - \dim Z + \dim W - \dim\cW \\
	&= \dim W - \dim Z.
\end{align*}
Therefore by miracle flatness for schemes, see e.g., \cite[26.2.11]{Vakil}, the map $W \to Z$ is flat. Since flatness can be checked smooth locally on the source, this implies that $\cW \to Z$ is flat. Then because flatness can be checked smooth locally on the target, this implies that $\cY \to \cZ$ is flat.
\end{proof}

We may now complete the proof of \autoref{theoremSmoothnessOfTruncation}.

\begin{proof}[Proof of \autoref{theoremSmoothnessOfTruncation}]
By \autoref{TwistedJetStacksAreSmoothAndEquidimensionalOfCorrectDimension}, each $\sJ_n(\cX)$ is smooth over $k$ and equidimension $(n+1)\dim\cX$. Therefore by \autoref{miracleFlatnessForStacks} we only need to show that for each $\ell$, the truncation map $\theta^{n+1}_n: \sJ_{n+1}^\ell(\cX) \to \sJ_n^\ell(\cX)$ has smooth fibers of dimension $\dim\cX$. The latter holds by \cite[Proposition 4.11]{SatrianoUsatine6}.
\end{proof}

\section{The compactly supported cohomology of a constructible subset}

We begin by recalling a method for canonically decomposing a constructible subset into locally closed subsets. In the case of schemes, this construction has appeared in e.g., \cite{Leykin}.

\begin{notation}\label{notationForCanonicalDecompostion}
Let $\cZ$ be a finite type Artin stack over $k$, and let $\cC \subset |\cZ|$ be a constructible subset. We will recursively define a sequence $\{\cY^\cC_i\}_{i \geq 0}$ of locally closed substacks of $\cZ$ as follows. First let $\cY^\cC_0$ be the unique reduced locally closed substack of $\cZ$ such that 
\[
	|\cY^\cC_0| = \overline{\cC} \setminus\left( \overline{\overline{\cC} \setminus \cC}\right).
\]
Then if we already have $\cY^\cC_0, \dots, \cY^\cC_r$, set $\cC_{r+1} = \cC \setminus \bigcup_{i=0}^r |\cY^\cC_i|$, and let $\cY^\cC_{r+1}$ be the unique reduced locally closed substack of $\cZ$ such that
\[
	|\cY^\cC_{r+1}| = \overline{\cC_{r+1}} \setminus\left( \overline{\overline{\cC_{r+1}} \setminus \cC_{r+1}}\right).
\]
Now having defined the sequence $\{\cY^\cC_i\}_{i \geq 0}$, we will let $\cY^\cC$ denote the disjoint union of the $\cY^\cC_i$.
\end{notation}

\begin{lemma}\label{canonicalDecompositionSmoothMap}
Let $\cW$ and $\cZ$ be finite type Artin stacks over $k$, let $\cC \subset |\cZ|$ be a constructible subset, let $\cW \to \cZ$ be a smooth morphism, and let $\cD$ be the preimage of $\cC$ in $|\cW|$. If $i \in \Z_{\geq 0}$, then $\cY^\cC_i \times_\cZ \cW$ and $\cY^\cD_i$ are isomorphic over $\cW$.
\end{lemma}

\begin{proof}
Because $\cW \to \cZ$ is smooth, the continuous map $|\cW| \to |\cZ|$ is open. Therefore taking preimage along this continuous map commutes with taking closure, so $|\cY^\cC_i \times_\cZ \cW| = |\cY^\cD_i|$ as subsets of $|\cW|$. The lemma then follows by the fact that $\cY^\cC_i \times_\cZ \cW$ is smooth over $\cY^\cC_i$ and therefore reduced.
\end{proof}

The following notation will be useful in what follows.

\begin{notation}
If $\cZ$ is a finite type Artin stack over $k$ with affine geometric stabilizers, and $\cC \subset |\cZ|$ is a constructible subset, we set
\[
	\dim\cC = \begin{cases} \log\Vert\e(\cC) \Vert, & \Vert\e(\cC) \Vert > 0 \\ -\infty, & \text{otherwise} \end{cases}.
\]
Note that by \autoref{RecallingNormAndDimension}, if $\{\cY_i\}_{i \geq 0}$ is a finite partition of $\cC$ into locally closed substacks of $\cZ$, then $\dim \cC = \dim \bigsqcup_{i \geq 0} \cY_i$. Furthermore, note that $\cC$ contains a dense open subset of $\overline{\cC}$, so $\dim \cC = \dim\overline{\cC}$.
\end{notation}

\begin{remark}
In the last sentence of the following proposition, the purpose of the assumption that $\cZ$ has affine geometric stabilizers is so that $\cY^\cC$ has affine geometric stabilizers and therefore has a well defined class in $\widehat{\sM}_k$.
\end{remark}

\begin{proposition}\label{propositionThatCanonicalDecompositionIsDecomposition}
Let $\cZ$ be a finite type Artin stack over $k$, and let $\cC \subset |\cZ|$ be a constructible subset. Then $\cY^\cC_{i} = \emptyset$ for all but finitely many $i$, and $\cC = \bigsqcup_{i \geq 0} |\cY^\cC_i|$. In particular, $\cY^\cC$ is finite type over $k$, and if $\cZ$ has affine geometric stabilizers, then $\e(\cC) = \e(\cY^\cC)$.
\end{proposition}

\begin{proof}
By \autoref{canonicalDecompositionSmoothMap}, we may assume that $\cZ$ is a scheme. By noetherian induction, we only need to show that $\cY^\cC_0 \subset \cC$ and that if $\cC \neq \emptyset$, then $\dim(\cC \setminus \cY^\cC_0) < \dim \cC$. The former is immediate from the construction, and the latter follows from the fact that $(\cC \setminus \cY^\cC_0) \subset \overline{\overline{\cC} \setminus \cC}$ and the fact that if $\cC \neq \emptyset$, then 
\[
	\dim\left(\overline{\overline{\cC} \setminus \cC}\right) = \dim\left( \overline{\cC} \setminus \cC \right) < \dim \overline{\cC} = \dim \cC,
\]
 where the above inequality is due to the fact that $\cC$ is a dense constructible subset of the closed subscheme $\overline{\cC}$ of $\cZ$.
\end{proof}

\begin{lemma}\label{lemmaConstructibleInclusionCanonicalSymmetricDifference}
Let $\cZ$ be a finite type Artin stack over $k$, and let $\cC \subset \cD \subset |\cZ|$ be constructible subsets. Then $|\cY^\cC_0| \setminus |\cY^\cD_0| \subset \overline{\cD \setminus \cC}$ and $|\cY^\cD_0| \setminus |\cY^\cC_0| \subset \overline{\cD \setminus \cC}$.
\end{lemma}

\begin{proof}
Set $\cC_0 = |\cY^\cC_0|$ and $\cD_0 = |\cY^\cD_0|$. Throughout this proof we will use that by \autoref{propositionThatCanonicalDecompositionIsDecomposition}, we have $\cC_0 \subset \cC$ and $\cD_0 \subset \cD$.

We will first show that $\cD_0 \setminus \cC_0 \subset \overline{\cD \setminus \cC}$. Since $\cD_0 \setminus \cC \subset \cD \setminus \cC \subset \overline{\cD \setminus \cC}$ and $(\cD_0 \setminus \cC_0) \setminus (\cD_0 \setminus \cC) = (\cC \setminus \cC_0) \cap \cD_0$, we only need to show that $(\cC \setminus \cC_0) \cap \cD_0 \subset \overline{\cD \setminus \cC}$. Thus let $x \in (\cC \setminus \cC_0) \cap \cD_0$, and let $U$ be an open neighborhood of $x$ in $|\cZ|$. We only need to show that $U \cap (\cD \setminus \cC) \neq \emptyset$. Recall that by definition of $\cY^\cD_0$, we have $\cD_0 = \overline{\cD} \setminus\left( \overline{\overline{\cD} \setminus \cD}\right)$, so $\cD_0$ is open in $\overline{\cD}$. Thus $U \cap \cD_0$ is open in $\overline{\cD}$, so there exists an open subset $U'$ of $|\cZ|$ such that $U' \cap \overline{\cD} = U \cap \cD_0$. Then since $x \in U \cap \cD_0$, the set $U'$ is an open neighborhood of $x$ in $|\cZ|$. Then since $x \in \cC \setminus \cC_0 \subset \overline{\cC} \setminus \cC_0 \subset \overline{\overline{\cC} \setminus \cC}$, where the latter inclusion is due to the fact that $\cC_0 = \overline{\cC} \setminus\left( \overline{\overline{\cC} \setminus \cC}\right)$, we have that $U' \cap (\overline{\cC} \setminus \cC) \neq \emptyset$. On the other hand
\begin{align*}
	U' \cap (\overline{\cC} \setminus \cC) &\subset U' \cap (\overline{\cD} \setminus \cC)\\
	&= (U' \cap \overline{\cD}) \cap (\overline{\cD} \setminus \cC)\\
	&= (U \cap \cD_0) \cap (\overline{\cD} \setminus \cC)\\
	&= U \cap (\cD_0 \setminus \cC)\\
	&\subset U \cap (\cD \setminus \cC),
\end{align*}
so $U \cap (\cD \setminus \cC) \neq \emptyset$, and we have finished showing that $\cD_0 \setminus \cC_0 \subset \overline{\cD \setminus \cC}$.

We will next show that $\cC_0 \setminus \cD_0 \subset \overline{\cD \setminus \cC}$. Since $\cC_0 = \overline{\cC} \setminus\left( \overline{\overline{\cC} \setminus \cC}\right)$, we have that $\cC_0$ is open in $\overline{\cC}$, so there exists an open subset $V'$ of $|\cZ|$ such that $V' \cap \overline{\cC} = \cC_0$. We will use later that $V' \setminus \overline{\cC} = V' \setminus (V' \cap \overline{\cC}) = V' \setminus \cC_0 \supset V' \setminus \cC$. Let $y \in \cC_0 \setminus \cD_0$, and let $V$ be an open neighborhood of $y$ in $|\cZ|$. We only need to show that $V \cap (\overline{\cD \setminus \cC})$ is nonempty. Since $V \cap V'$ is an open neighborhood of $y$, and $y \in \cC \setminus \cD_0 \subset \cD \setminus \cD_0 \subset \overline{\cD} \setminus \cD_0 \subset \overline{\overline{\cD} \setminus \cD}$, where the last inclusion is due to the fact that $\cD_0 = \overline{\cD} \setminus\left( \overline{\overline{\cD} \setminus \cD}\right)$, we have that $V \cap V' \cap (\overline{\cD} \setminus \cD)$ is nonempty. On the other hand
\begin{align*}
	V \cap V' \cap (\overline{\cD} \setminus \cD) &\subset V \cap (V' \setminus \cC) \cap \overline{\cD}\\
	&\subset V \cap (V' \setminus \overline{\cC}) \cap \overline{\cD}\\
	&= V \cap V' \cap (\overline{\cD} \setminus \overline{\cC})\\
	&\subset V \cap (\overline{\cD} \setminus \overline{\cC})\\
	&\subset V \cap (\overline{\cD \setminus \cC}),
\end{align*}
so $V \cap (\overline{\cD \setminus \cC})$ is nonempty, and we have finished showing that $\cC_0 \setminus \cD_0 \subset \overline{\cD \setminus \cC}$.
\end{proof}

\begin{proposition}\label{differenceOfCanonicalDecompositionPieces}
Let $\cZ$ be a finite type Artin stack over $k$, and let $\cC, \cD \subset |\cZ|$ be constructible subsets. Then for all $i \in \Z_{\geq 0}$, there exists an Artin stack $\cY_i$ such that $\cY_i$ is both a locally closed substack of $\cY^\cC_i$ and a locally closed substack of $\cY^\cD_i$ and $|\cY^\cC_i| \setminus |\cY_i| \subset (\overline{\cC \setminus \cD}) \cup (\overline{\cD \setminus \cC})$ and $|\cY^\cD_i| \setminus |\cY_i| \subset (\overline{\cC \setminus \cD}) \cup (\overline{\cD \setminus \cC})$.
\end{proposition}

\begin{proof}

Set $\cC_0 = |\cY^\cC_0|$ and $\cD_0 = |\cY^\cD_0|$. Throughout this proof we will use that by \autoref{propositionThatCanonicalDecompositionIsDecomposition}, we have $\cC_0 \subset \cC$ and $\cD_0 \subset \cD$.

We will first show that $(\cC_0 \setminus \cD_0) \cup (\cD_0 \setminus \cC_0) \subset (\overline{\cC \setminus \cD}) \cup (\overline{\cD \setminus \cC})$. Set $\cE = \cC \cap \cD$ and $\cE_0 = |\cY^\cE_0|$. Then by \autoref{lemmaConstructibleInclusionCanonicalSymmetricDifference}, we have $(\cE_0 \setminus \cC_0) \cup (\cC_0 \setminus \cE_0) \subset \overline{\cC \setminus \cE}$ and $(\cE_0 \setminus \cD_0) \cup (\cD_0 \setminus \cE_0) \subset \overline{\cD \setminus \cE}$. Therefore
\begin{align*}
	\cC_0 \setminus \cD_0 &\subset (\cC_0 \setminus \cE_0) \cup (\cE_0 \setminus \cD_0)\\
	&\subset (\overline{\cC \setminus \cE}) \cup (\overline{\cD \setminus \cE})\\
	&= (\overline{\cC \setminus \cD}) \cup (\overline{\cD \setminus \cC}),
\end{align*}
and similarly $\cD_0 \setminus \cC_0 \subset (\overline{\cC \setminus \cD}) \cup (\overline{\cD \setminus \cC})$, as desired.

Now since $\cC_0 \cap \cD_0$ is a locally closed subset of $|\cZ|$, we may let $\cY_0$ be the reduced locally closed substack of $\cZ$ such that $|\cY_0| = \cC_0 \cap \cD_0$. Then $\cY_0$ is both a locally closed substack of $\cY^\cC_0$ and a locally closed substack of $\cY^\cD_0$ and $|\cY^\cC_0| \setminus |\cY_0| = \cC_0 \setminus \cD_0 \subset (\overline{\cC \setminus \cD}) \cup (\overline{\cD \setminus \cC})$ and $|\cY^\cD_0| \setminus |\cY_0| = \cD_0 \setminus \cC_0 \subset (\overline{\cC \setminus \cD}) \cup (\overline{\cD \setminus \cC})$.

Now set $\cC' = \cC \setminus \cC_0$ and $\cD' = \cD \setminus \cD_0$. By induction, we only need to show that $(\overline{\cC' \setminus \cD'}) \cup (\overline{\cD' \setminus \cC'}) \subset (\overline{\cC \setminus \cD}) \cup (\overline{\cD \setminus \cC})$. By symmetry, we only need to verify that $\cC' \setminus \cD' \subset (\overline{\cC \setminus \cD}) \cup (\overline{\cD \setminus \cC})$. Now let $x \in \cC' \setminus \cD'$. Either $x \notin \cD$ or $x \in \cD$. In the first case, $x \in \cC' \setminus \cD \subset \cC \setminus \cD \subset (\overline{\cC \setminus \cD}) \cup (\overline{\cD \setminus \cC})$. In the second case, $x \in \cD \setminus \cD' = \cD_0$, so $x \in \cC' \cap \cD_0 = (\cC \setminus \cC_0) \cap \cD_0 \subset (\cD_0 \setminus \cC_0) \subset (\overline{\cC \setminus \cD}) \cup (\overline{\cD \setminus \cC})$, and we are done.
\end{proof}

\begin{corollary}\label{cohomologyOfConstructibleSetsWithSmallDifference}
Let $\cZ$ be a finite type Artin stack over $\C$ with affine geometric stabilizers, and let $\cC \subset \cD \subset |\cZ|$ be constructible subsets. If $i > 1 + 2\dim(\cD \setminus \cC)$, then
\[
	H^i_c(\cY^\cC) \cong H^i_c(\cY^\cD).
\]
\end{corollary}

\begin{proof}
It is sufficient to show that for all $j \in \Z_{\geq 0}$, we have $H_c^i(\cY^\cC_j) \cong H_c^i(\cY^\cD_j)$. By \autoref{differenceOfCanonicalDecompositionPieces}, there exists some $\cY_j$ that is a locally closed substack of both $\cY_j^\cC$ and $\cY^\cD_j$ and $|\cY^\cC_j| \setminus |\cY_j| \subset (\overline{\cD \setminus \cC})$ and $|\cY^\cD_j| \setminus |\cY_j| \subset (\overline{\cD \setminus \cC})$. Thus by \autoref{containedConstructiblesBound},
\[
	\Vert \e(|\cY^\cC_j| \setminus |\cY_j|) \Vert \leq \Vert \e(\overline{\cD \setminus \cC}) \Vert,
\]
so
\[
	\dim(|\cY^\cC_j| \setminus |\cY_j|) \leq \dim(\overline{\cD \setminus \cC}) = \dim(\cD \setminus \cC).
\]
Similarly,
\[
	\dim(|\cY^\cD_j| \setminus |\cY_j|) \leq \dim(\cD \setminus \cC).
\]
Since it is sufficient to prove that $H_c^i(\cY^\cC_j) \cong H_c^i(\cY_j)$ and $H_c^i(\cY^\cD_j) \cong H_c^i(\cY_j)$, the corollary follows from the next lemma.
\end{proof}

\begin{lemma}
Let $\cZ$ be a finite type Artin stack over $\C$ with affine geometric stabilizers, and let $\cW$ be a locally closed substack of $\cZ$. If $i > 1+ 2\dim(|\cZ| \setminus |\cW|)$, then $H^i_c(\cZ) \cong H^i_c(\cW)$.
\end{lemma}

\begin{proof}
The locally closed immersion $\cW \hookrightarrow \cZ$ factors as an open immersion $\cW \hookrightarrow \cV$ followed by a closed immersion $\cV \hookrightarrow \cZ$. 

We will first show that $H_c^i(\cW) \cong H_c^i(\cV)$. Let $\cY$ be a closed subscheme of $\cV$ supported on $|\cV| \setminus |\cW|$. Then we have an exact sequence as follows.
\[
	H_c^{i-1}(\cY) \to H_c^i(\cW) \to H_c^i(\cV) \to H_c^i(\cY).
\]
We also have that
\[
	\dim(\cY) = \dim(|\cY|) = \dim(|\cV| \setminus |\cW|) \leq \dim(|\cZ| \setminus |\cW|),
\]
where the inequality is by \autoref{containedConstructiblesBound}. Therefore $i - 1 > 2\dim(\cY)$, so $H_c^{i-1}(\cY) \cong H_c^i(\cY) \cong 0$. Thus $H_c^i(\cW) \cong H_c^i(\cV)$.

We now only need to prove that $H_c^i(\cZ) \cong H_c^i(\cV)$. Let $\cU$ be the open substack of $\cZ$ supported on $|\cZ| \setminus |\cV|$. Then we have an exact sequence as follows.
\[
	H_c^i(\cU) \to H_c^i(\cZ) \to H_c^i(\cV) \to H_c^{i+1}(\cU).
\]
We also have that
\[
	\dim(\cU) = \dim(|\cU|) = \dim(|\cZ| \setminus |\cV|) \leq \dim(|\cZ| \setminus |\cW|),
\]
where again the inequality is by \autoref{containedConstructiblesBound}. Thus $i > 2\dim(\cU)$, so $H_c^i(\cU) \cong H_c^{i+1}(\cU) \cong 0$. Therefore $H_c^i(\cZ) \cong H_c^i(\cV)$, and we are done.
\end{proof}

\section{The co-unit map of a stacky affine space fibration}

The main goal of this section is to prove the following.

\begin{theorem}\label{theoremGysinIsIsomorphism}
Let $\cY$ and $\cZ$ be finite type Artin stacks over $\C$, and let $f: \cY \to \cZ$ be a smooth morphism of constant relative dimension. Furthermore assume that for every field extension $k'$ of $\C$ and every morphism $\Spec(k') \to \cZ$, there exists $r,s \in \Z_{\geq 0}$ such that
\[
	\cY \times_\cZ \Spec(k') \cong \bA^r_{k'} \times_{k'} B\bG_{a,k'}^s
\]
as stacks over $k'$. Then the co-unit map $f_!f^!\Q_\cZ \to \Q_\cZ$ is an isomorphism. In particular, if $d$ is the relative dimension of $f$, then
\[
	H_c^{i+2d}(\cY)(d) \cong H_c^i(\cZ)
\]
for all $i \in \Z$.
\end{theorem}

Before the proof, we show that \autoref{theoremGysinIsIsomorphism} implies the following corollary, which is a version of \cite[Theorem 2.5(ii)]{Ekedahl}, but in the setting of mixed Hodge structures, that is needed for our purposes.

\begin{corollary}\label{ChiHdgDoesWhatWeExpect}
Let $\cZ$ be a finite type Artin stack over $\C$ with affine geometric stabilizers. Then $\sum_{i \in \Z} (-1)^i[H_c^i(\cZ)]$ converges in $\widehat{K_0(\MHS)}$, and
\[
	\chi_\Hdg(\e(\cZ)) = \sum_{i \in \Z} (-1)^i[H_c^i(\cZ)].
\]
\end{corollary}

\begin{proof}
The series $\sum_{i \in \Z} (-1)^i[H_c^i(\cZ)]$ converges in $\widehat{K_0(\MHS)}$ by \autoref{cohomologyVanishingAboveTwiceDimension} and \autoref{pqPieceVanishesAboveDegree}. By \autoref{excisionExactSequence}, \autoref{theoremGysinIsIsomorphism}, and the definition of $K_0(\Stack_\C)$, there is a well defined ring homomorphism 
\[
	\chi: K_0(\Stack_\C) \to \widehat{K_0(\MHS)}
\]
that takes the class of any finite type Artin stack $\cY$ over $\C$ with affine geometric stabilizers to $\sum_{i \in \Z}(-1)^i[H_c^i(\cY)]$. Therefore we only need to show that $\chi$ coincides with the composition 
\[
	\chi': K_0(\Stack_\C) \to \widehat{\sM}_\C \xrightarrow{\chi_\Hdg} \widehat{K_0(\MHS)}.
\]
But this follows from the fact that the composition 
\[
	K_0(\Var_\C) \to K_0(\Stack_\C) \xrightarrow{\chi} \widehat{K_0(\MHS)}
\]
coincides with the composition 
\[
	K_0(\Var_\C) \to K_0(\Stack_\C) \xrightarrow{\chi'} \widehat{K_0(\MHS)}
\]
and the fact that $K_0(\Var_\C) \to K_0(\Stack_\C)$ is a localization by \cite[Theorem 1.2]{Ekedahl}.
\end{proof}

The remainder of this section will be dedicated to proving \autoref{theoremGysinIsIsomorphism}. We begin with some lemmas.

\begin{lemma}\label{lemmaCompositionOfCoUnitMaps}
Let $\cY, \cZ$, and $\cW$ be locally finite type Artin stacks over $\C$, let $f: \cY \to \cZ$ be a smooth morphism of constant relative dimension, and let $g: \cW \to \cY$ be a morphism. If the co-unit maps $g_! g^! \Q_\cY \to \Q_\cY$ and $(f \circ g)_! (f \circ g)^! \Q_\cZ \to \Q_\cZ$ are isomorphisms, then the co-unit map $f_! f^! \Q_\cZ \to \Q_\cZ$ is an isomorphism.
\end{lemma}

\begin{proof}
Let $d$ be the relative dimension of $f$. Then applying the purity isomorphism $f^*(d)[2d] \xrightarrow{\sim} f^!$ \cite[Theorem 3.1(7)]{Tubach2} to the isomorphism $g_! g^! f^* \Q_\cZ(d)[2d] = g_! g^! \Q_\cY (d)[2d] \xrightarrow{\sim} \Q_\cY(d)[2d] = f^*\Q_\cZ (d)[2d]$ gives that the co-unit map $g_!g^! f^!\Q_\cZ \to f^! \Q_\cZ$ is an isomorphism, so applying $f_!$ gives that the map $(f \circ g)_! (f \circ g)^! \Q_\cZ \to f_!f^! \Q_\cZ$ is an isomorphism. The lemma then follows from the fact that (in the homotopy category of $D_H(\cZ)$) the co-unit map $(f \circ g)_! (f \circ g)^! \Q_\cZ \to \Q_\cZ$ is the composition of $(f \circ g)_! (f \circ g)^! \Q_\cZ \to f_!f^! \Q_\cZ$ with the co-unit map $f_! f^! \Q_\cZ \to \Q_\cZ$.
\end{proof}

\begin{lemma}\label{lemmaCoUnitAfterPulback}
Let $\cY, \cZ$, and $\cW$ be locally finite type Artin stacks over $\C$, let $f: \cY \to \cZ$ be a smooth morphism of constant relative dimension, let $\cW \to \cZ$ be a morphism, and let $g$ be the base change of $f$ along $\cW \to \cZ$.
\begin{enumerate}

\item\label{partOfLemmaWhereCoUnitOfPullbackIsIso} If the co-unit map $f_!f^! \Q_\cZ \to \Q_\cZ$ is an isomorphism, then the co-unit map $g_!g^!\Q_\cW \to \Q_\cW$ is an isomorphism.

\item\label{partOfCoUnitPullbackLemmaTheOtherWay} If $\cW \to \cZ$ is a smooth cover and the co-unit map $g_!g^!\Q_\cW \to \Q_\cW$ is an isomorphism, then the co-unit map $f_!f^! \Q_\cZ \to \Q_\cZ$ is an isomorphism.

\end{enumerate}
\end{lemma}

\begin{proof}
Let $d$ be the relative dimension of $f$, let $h$ be the map $\cW \to \cZ$, let $\cV = \cY \times_\cZ \cW$, and let $q$ be the projection $\cV \to \cY$. For both parts, it is sufficient to prove that applying $h^*$ to $f_!f^! \Q_\cZ \to \Q_\cZ$ gives (up to isomorphism) $g_!g^!\Q_\cW \to \Q_\cW$. This follows from the canonical isomorphisms 
\begin{align*}
	h^*f_!f^!\Q_\cZ &\cong g_!q^*f^!\Q_\cZ \cong g_! q^* f^* \Q_\cZ(d)[2d] \\
	&\cong g_!g^*h^*\Q_\cZ(d)[2d] \cong g_!g^!h^*\Q_\cZ \cong g_!g^! \Q_\cW,
\end{align*}
where the first isomorphism is base change \cite[Theorem 3.1(4)]{Tubach2}, and the second and fourth isomorphisms are purity \cite[Theorem 3.1(7)]{Tubach2}.
\end{proof}

\begin{lemma}\label{lemmaOpenClosedCoUnit}
Let $\cY$ be a locally finite type Artin stack over $\C$, let $\cZ$ be a closed substack of $\cY$, and let $\cU$ be its open complement. Let $\cW$ be a locally finite type Artin stack over $\C$, let $f: \cW \to \cY$ be a smooth morphism of constant relative dimension, and let $g$ and $h$ be the base changes of $f$ to $\cU$ and $\cZ$, respectively. If the co-unit maps $g_! g^! \Q_\cU \to \Q_\cU$ and $h_! h^! \Q_\cZ \to \Q_\cZ$ are isomorphisms, then the co-unit map $f_! f^! \Q_\cY \to \Q_\cY$ is an isomorphism.
\end{lemma}

\begin{proof}
Let $j$ and $i$ be the inclusions of $\cU$ and $\cZ$, respectively, into $\cY$. Then \autoref{excisionExactSequence} and \cite[Theorem 3.1(3 and 7)]{Tubach2} gives for each $\cK \in D_H(\cY)$ a canonical (functorial in $\cK$) distinguished triangle
\[
	j_!j^*\cK \to \cK \to i_!i^*\cK.
\]
Thus by the 5-lemma and the fact that $D_H(\cY)$ has a t-structure, it is sufficient to show that $j^*$ (resp. $i^*$) applied to the co-unit map $f_! f^! \Q_\cY \to \Q_\cY$ gives (up to isomorphism) the co-unit map $g_! g^! \Q_\cU \to \Q_\cU$ (resp. $h_! h^! \Q_\cZ \to \Q_\cZ$), but this was shown in the proof of \autoref{lemmaCoUnitAfterPulback}.
\end{proof}

We are now equipped to prove \autoref{theoremGysinIsIsomorphism} by reducing it to a sequence of special cases.

\begin{proposition}\label{propCoUnitIsoLocallyAffineBundle}
Let $\cY$ and $\cZ$ be locally finite type Artin stacks over $\C$, and let $f: \cY \to \cZ$ be a morphism. If there exists some $r \in \Z_{\geq 0}$ and a smooth cover $\cW \to \cZ$ by an Artin stack $\cW$ such that $\cY \times_\cZ \cW$ is isomorphic (as a stack over $\cW$) to $\bA^r_\C \times_\C \cW$, then the co-unit map $f_! f^! \Q_\cZ \to \Q_\cZ$ is an isomorphism.
\end{proposition}

\begin{proof}
Let $g: \bA^r_\C \to \Spec(\C)$ be the structure map. For each $i \in \Z$, applying $H^i$ to the co-unit map $g_! g^! \Q \to \Q$ gives a map $H^{i+2r}_c(\bA^r_\C)(r) \to H^i_c(\Spec(\C))$ that is dual to the pullback map $H^i(\Spec(\C)) \to H^i(\bA^r_\C)$, which is an isomorphism since $\bA^r_\C$ is contractible. Therefore the co-unit map $g_! g^! \Q \to \Q$ is an isomorphism.

Now let $h: \cY \times_\cZ \cW \to \cW$ be the projection. Then by the hypotheses on $\cY \times_\cZ \cW$, \autoref{lemmaCoUnitAfterPulback}(\ref{partOfLemmaWhereCoUnitOfPullbackIsIso}), and the fact that $g_! g^! \Q \to \Q$ is an isomorphism, we have that the co-unit map $h_!h^! \Q_{\cW} \to \Q_{\cW}$ is an isomorphism. Therefore $f_! f^! \Q_\cZ \to \Q_\cZ$ is an isomorphism by \autoref{lemmaCoUnitAfterPulback}(\ref{partOfCoUnitPullbackLemmaTheOtherWay}).
\end{proof}

\begin{proposition}\label{propLocallyStackyAffineBundleCoUnit}
Let $\cY$ and $\cZ$ be locally finite type Artin stacks over $\C$, and let $f: \cY \to \cZ$ be a morphism. If there exists some $r,s  \in \Z_{\geq 0}$ and a smooth cover $\cW \to \cZ$ by an Artin stack $\cW$ such that $\cY \times_\cZ \cW$ is isomorphic (as a stack over $\cW$) to $\bA^r_\C \times_\C B\bG_{a,\C}^s \times_\C \cW$, then the co-unit map $f_! f^! \Q_\cZ \to \Q_\cZ$ is an isomorphism.
\end{proposition}

\begin{proof}
Let $g:  \cY \times_\cZ \cW \to \cW$ be the projection. Then by the hypotheses on $\cY \times_\cZ \cW$, we have a $\cW$-morphism $h: \bA^r_\C \times_\C \cW \to \cY \times_\cZ \cW$ that is a $\bG_{a,\C}^s$-torsor. By \autoref{propCoUnitIsoLocallyAffineBundle}, the co-unit maps $h_! h^! \Q_{\cY \times_\cZ \cW} \to \Q_{\cY \times_\cZ \cW}$ and $(g \circ h)_! (g \circ h)^! \Q_\cW \to \Q_\cW$ are isomorphisms. Thus by \autoref{lemmaCompositionOfCoUnitMaps}, the co-unit map $g_! g^! \Q_\cW \to \Q_\cW$ is an isomorphism. This implies the proposition by \autoref{lemmaCoUnitAfterPulback}(\ref{partOfCoUnitPullbackLemmaTheOtherWay}).
\end{proof}

\begin{proof}[Proof of \autoref{theoremGysinIsIsomorphism}]
We begin by noting that $f_! f^! \Q_\cZ \to \Q_\cZ$ being an isomorphism implies the last sentence of \autoref{theoremGysinIsIsomorphism} by the purity isomorphism \cite[Theorem 3.1(7)]{Tubach2}, so we only need to show that $f_! f^! \Q_\cZ \to \Q_\cZ$ is an isomorphism.

By \autoref{lemmaOpenClosedCoUnit}, we may replace $f$ with its base change along $\cZ_\red \to \cZ$ and therefore assume $\cZ$ is reduced. By noetherian induction, \autoref{lemmaOpenClosedCoUnit}, and \autoref{propLocallyStackyAffineBundleCoUnit}, it is sufficient to show that there exists a nonempty open substack $\cU$ of $\cZ$ over which $f$ satisfies the hypotheses of \autoref{propLocallyStackyAffineBundleCoUnit}. Let $V \to \cZ$ be a smooth cover by a scheme $V$. Since $\cZ$ is finite type over $\C$, we may assume that $V$ is finite type over $\C$. Let $\cW \to V$ be the base change of $f$ along $V \to \cZ$. Then by the hypotheses on $f$ and \cite[Lemma 2.3]{SatrianoUsatine6}, there exist $r, s \in \Z_{\geq 0}$ and a nonempty open subset $U$ of $V$ such that $(\cW \times_V U)_\red \cong (\bA^r_\C \times_\C U)_\red$ as stacks over $V$. Note that \cite[Lemma 2.3]{SatrianoUsatine6} only states that the isomorphism $(\cW \times_V U)_\red \cong (\bA^r_\C \times_\C B\bG_{a,\C}^s \times_\C U)_\red$ is over $k$, but the proof, which can be found as the main step in the proof of \cite[Proposition 2.8]{SatrianoUsatine1}, actually gives that the isomorphism is as stacks over $V$. Since $\cZ$ is reduced, $V$ is reduced, so $U$ is reduced. Noting also that $\cW \times_V U \to U$ is smooth by the fact that $f$ is smooth, we therefore get that $\cW \times_V U$ and $\bA^r_\C \times_\C B\bG_{a,\C}^s \times_\C U$ are both reduced, so in fact $\cW \times_V U \cong \bA^r_\C \times_\C B\bG_{a,\C}^s \times_\C U$ as stacks over $V$. Since smooth morphisms are open, we may let $\cU$ be the open substack of $\cZ$ supported on the image of $U$ in $|\cZ|$. Then $\cU$ is a nonempty open substack of $\cZ$ over which $f$ satisfies the hypotheses of \autoref{propLocallyStackyAffineBundleCoUnit}, and we are done.
\end{proof}

\section{A cohomological interpretation for certain measurable sets}

Throughout this section, let $\cX$ be a smooth finite type equidimensional Artin stack over $\C$ with affine diagonal and a good moduli space, and let $\cA \subset |\sJ(\cX)|$ be a measurable set such that $\theta_n(\cA)$ is a quasi-compact locally constructible subset of $|\sJ_n(\cX)|$ for all $n \in \Z_{\geq 0}$ and $\nu_\cX(\cA) = \lim_{n \to \infty}(\nu_\cX(\theta_n^{-1}(\theta_n(\cA))))$. In this section, we will give a cohomological interpretation for the image of $\nu_\cX(\cA)$ in $\widehat{K_0(\MHS)}$. We begin with the following theorem.

\begin{theorem}\label{cohomologyForSpecialMeasurableStabilizes}
Let $i \in \Z$. Then there exists some $m \in \Z_{\geq 0}$ such that for all $n \geq m$,
\[
	H_c^{i+2n\dim\cX}(\cY^{\theta_n(\cA)})(n\dim\cX) \cong H_c^{i+2m\dim\cX}(\cY^{\theta_m(\cA)})(m\dim\cX).
\]
\end{theorem}

\begin{proof}
For each $n \in \Z_{\geq 0}$, set $\cC_n = \theta_n(\cA)$. Then each $\cC_n$ is a quasi-compact locally constructible subset of $|\sJ_n(\cX)|$ and the sequence $\{\nu_\cX(\theta_n^{-1}(\cC_n))\}_n$ converges as $n \to \infty$. Therefore there exists some $m \in \Z_{\geq 0}$ such that for all $n \geq m$,
\[
	\Vert \nu_\cX(\theta_m^{-1}(\cC_m)) - \nu_\cX(\theta_n^{-1}(\cC_n)) \Vert < \exp(-\dim\cX + (i-1)/2).
\]
For the remainder of this proof, let $n \geq m$, and set $\cC = \cC_n$ and $\cD = (\theta^n_m)^{-1}(\cC_m)$. Then
\begin{align*}
	\nu_\cX(\theta_m^{-1}(\cC_m)) - \nu_\cX(\theta_n^{-1}(\cC_n)) &= \left(\e(\cD) - \e(\cC)\right)\bL^{-(n+1)\dim\cX}\\
	&= \e(\cD \setminus \cC)\bL^{-(n+1)\dim\cX},
\end{align*}
where the second equality is because $\cC \subset \cD$. Therefore
\begin{align*}
	\dim(\cD \setminus \cC) &= \log\Vert \e(\cD \setminus \cC) \Vert \\
	&= \log\left(\Vert \bL^{(n+1)\dim\cX} \Vert \cdot \Vert \e(\cD \setminus \cC)\bL^{-(n+1)\dim\cX} \Vert\right)\\
	&= (n+1)\dim\cX + \log\Vert \e(\cD \setminus \cC)\bL^{-(n+1)\dim\cX} \Vert\\
	&< (n+1)\dim\cX - \dim\cX + (i-1)/2,
\end{align*}
where the third equality is by \autoref{RecallingNormAndDimension}. Rearranging gives
\[
	i + 2n\dim\cX > 1+ 2\dim(\cD \setminus \cC).
\]
Therefore by \autoref{cohomologyOfConstructibleSetsWithSmallDifference},
\[
	H_c^{i + 2n\dim\cX}(\cY^\cC) \cong H_c^{i + 2n\dim\cX}(\cY^\cD).
\]
Therefore in order to complete our proof, it is sufficient to prove that 
\[
	H_c^{i + 2n\dim\cX}(\cY^\cD)(n\dim\cX) \cong H^{i+2m\dim\cX}(\cY^{\cC_m})(m\dim\cX).
\]
For each $r \geq m$, set $\cD_r = (\theta^r_m)^{-1}(\cC_m)$. In particular $\cD_n = \cD$ and $\cD_m = \cC_m$. Thus it is sufficient to show that for all $r \geq m$,
\[
	H_c^{i + 2(r+1)\dim\cX}(\cY^{\cD_{r+1}})((r+1)\dim\cX) \cong H_c^{i + 2r\dim\cX}(\cY^{\cD_r})(r\dim\cX).
\]
We have that $\cD_{r+1} = (\theta^{r+1}_r)^{-1}(\cD_r)$. Therefore by \autoref{theoremSmoothnessOfTruncation} and \autoref{canonicalDecompositionSmoothMap}, there is a morphism $\cY^{\cD_{r+1}} \to \cY^{\cD_r}$ that is a base change of $\theta^{r+1}_r: \sJ_{r+1}(\cX) \to \sJ_r(\cX)$. Thus by \autoref{theoremSmoothnessOfTruncation}, \cite[Proposition 4.11]{SatrianoUsatine6}, and \autoref{theoremGysinIsIsomorphism},
\[
	H_c^{i + 2(r+1)\dim\cX}(\cY^{\cD_{r+1}})(\dim\cX) \cong H_c^{i + 2r\dim\cX}(\cY^{\cD_r}),
\]
and we are done.
\end{proof}

We are now able to set the following notation.

\begin{notation}
For $i \in \Z$, let $H^i(\cA)$ denote $H_c^{i+2m\dim\cX}(\cY^{\theta_m(\cA)})(m\dim\cX)$, where $m$ satisfies the conclusion of \autoref{cohomologyForSpecialMeasurableStabilizes}. By \autoref{cohomologyForSpecialMeasurableStabilizes}, we have that $H^i(\cA)$ is a well defined (up to isomorphism) polarizable mixed Hodge structure.
\end{notation}

\begin{remark}\label{remarkAboutCohomologyVanishingAndConvergence}
By definition, we have $h^{p,q}(H^i(\cA)) = 0$ if $p+q > i$. By \autoref{theoremSmoothnessOfTruncation}, we have $\dim(\cY^{\theta_n(\cA)}) \leq \dim \sJ_n(\cX) = (n+1)\dim\cX$ for any $n \in \Z$, so $H^i(\cA) = 0$ for all $i > 2\dim\cX$. In particular, the series $\sum_{i \in \Z} (-1)^i [H^i(\cA)]$ converges in $\widehat{K_0(\MHS)}$.
\end{remark}

We will now prove that the $H^i(\cA)$ give a cohomological interpretation for the image of $\nu_\cX(\cA)$ in $\widehat{K_0(\MHS)}$.

\begin{theorem}\label{theoremCohomologicalInterpretationForCertainMeasurableSets}
We have
 \[
 	\chi_\Hdg\left(\bL^{\dim\cX}\nu_\cX(\cA)\right) = \sum_{i \in \Z} (-1)^i [H^i(\cA)].
\]
\end{theorem}

\begin{proof}
Set $\Theta = \chi_\Hdg\left(\bL^{\dim\cX}\nu_\cX(\cA)\right)$, and let $\varepsilon \in \R_{>0}$. It is sufficient to show that $\Vert \Theta - \sum_{i \in \Z} (-1)^i [H^i(\cA)] \Vert < \varepsilon$.

For each $n \in \Z_{\geq 0}$, set $\cC_n = \theta_n(\cA)$. Since $\nu_\cX(\cA) = \lim_{n \to \infty} \nu_{\cX}(\theta^{-1}_n(\cC_n))$, there exists some $m_1 \in \Z_{\geq 0}$ such that for all $n \geq m_1$,
\[
	\Vert \nu_\cX(\cA) - \e(\cY^{\cC_n})\bL^{-(n+1)\dim\cX} \Vert = \Vert \nu_{\cX}(\cA) - \nu_{\cX}(\theta^{-1}_n(\cC_n)) \Vert < \varepsilon/\exp(\dim\cX),
\]
where the first equality is by \autoref{propositionThatCanonicalDecompositionIsDecomposition} and \cite[Theorem 4.12]{SatrianoUsatine6}. Recalling that $\Vert\Psi\Vert\leq\Vert\chi_\Hdg(\Psi)\Vert\leq$ for all $\Psi\in \widehat{\sM}_\C$, it follows from \autoref{ChiHdgDoesWhatWeExpect} that
\[
	\Vert \Theta - \sum_{i \in \Z} (-1)^i[H_c^{i+2n\dim\cX}(\cY^{\theta_n(\cA)})(n\dim\cX)] \Vert < \varepsilon.
\]
By \autoref{remarkAboutCohomologyVanishingAndConvergence}, there are only finitely many $i \geq 2\log(\varepsilon)$ such that $H^i(\cA) \neq 0$, so there exists some $m_2 \in \Z_{\geq 0}$ such that
\[
	H^i(\cA) \cong H_c^{i+2n\dim\cX}(\cY^{\theta_n(\cA)})(n\dim\cX)
\]
for all $n \geq m_2$ and $i \geq 2\log(\varepsilon)$. Now set $m = \max(m_1, m_2)$. By \autoref{remarkAboutCohomologyVanishingAndConvergence},
\[
	\Vert [H^i(\cA)] \Vert \leq \exp(i/2)
\]
and
\[
	\Vert [H_c^{i+2m\dim\cX}(\cY^{\theta_m(\cA)})(m\dim\cX)] \Vert \leq \exp(i/2)
\]
for all $i$. Therefore
\[
	\Vert \left(\sum_{i \in \Z} (-1)^i [H^i(\cA)]\right) - \left(\sum_{i \in \Z} (-1)^i[H_c^{i+2m\dim\cX}(\cY^{\theta_m(\cA)})(m\dim\cX)]\right) \Vert < \varepsilon.
\]
Since $m \geq m_1$, we have
\[
	\Vert \Theta - \sum_{i \in \Z} (-1)^i[H_c^{i+2m\dim\cX}(\cY^{\theta_m(\cA)})(m\dim\cX)] \Vert < \varepsilon,
\]
so together we have
\[
	\Vert \Theta - \sum_{i \in \Z} (-1)^i [H^i(\cA)] \Vert < \varepsilon,
\]
completing our proof.
\end{proof}

We end this section by noting that we have now completed the proofs of \autoref{stringyCohomologyStabilizes} and \autoref{theoremStringyCohomologyGivesStringyHD}. Indeed the former is a direct consequence of \autoref{theoremStringyClassWellDefined} and \autoref{cohomologyForSpecialMeasurableStabilizes}, and the latter is a direct consequence of \autoref{theoremStringyClassWellDefined} and \autoref{theoremCohomologicalInterpretationForCertainMeasurableSets}.

\section{Orbifold cohomology}

As in \autoref{sectionSmoothnessOfTruncationMorphisms}, if $n \in \Z_{\geq 0}$, $\ell \in \Z_{>0}$, and $A$ is a $k$-algebra, we will let $\cD^\ell_{n, A}$ denote the stack quotient $[\Spec( A[t^{1/\ell}]/(t^{n+1}) ) / \mu_\ell]$, where $\xi \in \mu_\ell$ acts on $A[t^{1/\ell}]/(t^{n+1})$ by $f(t^{1/\ell}) \mapsto f(\xi t^{1/\ell})$. Throughout this section, let $\cX$ be a smooth finite type equidimensional tame Artin stack over $k$ with affine diagonal. For each $\ell \in \Z_{>0}$, the closed immersion $B\mu_\ell = (\cD^\ell_{0,k})_\red \hookrightarrow \cD^\ell_{0,k}$ induces a morphism $\pi_\ell: \sJ_0^\ell(\cX) \to I_{\mu_\ell}(\cX)$. The key input to proving \autoref{theoremStringyIsOrbifold} is the following proposition, which we will dedicate most of this section to proving.

\begin{proposition}\label{propComparingCohomologyInertiaToZerothJet}
Let $\ell \in \Z_{>0}$, and let $\cY$ be a connected component of $I_{\mu_\ell}(\cX)$. If $k = \C$, then for all $i \in \Z$,
\[
	H^{i+2(\dim\cX - \dim\cY)}_c(\pi_\ell^{-1}(\cY))(\dim\cX - \dim\cY) \cong H^i_c(\cY).
\]
\end{proposition}

We need some notation that we will use for the remainder of this section. First note that since $k$ has characteristic 0 and $\cX$ is smooth and tame, $\cX$ is smooth and Deligne--Mumford and therefore has a tangent bundle $T_\cX$. If $\ell \in \Z_{>0}$, $\cY$ is a connected component of $I_{\mu_\ell}(\cX)$, and $a \in \Z$, we will set 
\[
	d_a(\cY) = \dim_{k'} H^0((f^* T_\cX)(-a)),
\]
where $f$ is any morphism $B\mu_\ell \otimes_k k' \to \cX$, with $k'$ a field extension of $k$, such that the corresponding point $\Spec(k') \to I_{\mu_\ell}(\cX)$ is in $\cY$. It is a standard fact that $d_a(\cY)$ does not depend on the choice of $f$. See e.g., the proof of \cite[Proposition 4.19]{SatrianoUsatine6}. If $\ell, n \in \Z_{>0}$, and $A$ is a $k$-algebra, we will set
\[
	\cC^\ell_{n,A} = [\Spec( A[t^{1/\ell}]/(t^{n/\ell}) ) / \mu_\ell],
\]
where $\xi \in \mu_\ell$ acts on $A[t^{1/\ell}]/(t^{n/\ell})$ by $f(t^{1/\ell}) \mapsto f(\xi t^{1/\ell})$. In particular, $\cC^\ell_{1,A} = B\mu_\ell \otimes_k A$ and $\cC^\ell_{\ell, A} = \cD^\ell_{0,A}$. For each $\ell, n \in \Z_{>0}$, we will set
\[
	\cZ^\ell_n = \uHom^{\rep}_k(\cC^\ell_{n,k}, \cX).
\]
In particular, $\cZ^\ell_1 = I_{\mu_\ell}(\cX)$ and $\cZ^\ell_\ell = \sJ^\ell_0(\cX)$.

\begin{remark}\label{remarkAboutHomRepSubJetStack}
Suppose that $n \leq \ell$. Then the structure map $\cC^\ell_{n,k} \to \Spec(k)$ is a good moduli space map with affine diagonal. Thus by \cite[Proposition 4.4]{SatrianoUsatine4}, $\cZ^\ell_n$ is an open substack of $\uHom_k(\cC^\ell_{n,k}, \cX)$, which is a finite type Artin stack over $k$ with affine diagonal by \cite[Theorem 3.12(xv, xii) and Remark 3.13]{Rydh}. Therefore $\cZ^\ell_n$ is a finite type Artin stack over $k$ with affine diagonal.
\end{remark}

For each $\ell, n \in \Z_{>0}$, the closed immersion $\cC^\ell_{n, k} \hookrightarrow \cC^\ell_{n+1, k}$ induces a map 
\[
	\cZ^\ell_{n+1} \to \cZ^\ell_n.
\]
If in addition $\cY$ is a connected component of $I_{\mu_\ell}(\cX)$, we will let $\cY^\ell_n$ denote the preimage of $\cY$ along the composition
\[
	\cZ^\ell_n \to \cZ^\ell_{n-1} \to \dots \to \cZ^\ell_1 = I_{\mu_\ell}(\cX).
\]

\begin{proposition}\label{propositionSubJetDimension}
Let $\ell, n \in \Z_{>0}$, assume $n \leq \ell$, and let $\cY$ be a connected component of $I_{\mu_\ell}(\cX)$. Then $\cY^\ell_n$ is smooth over $k$ and equidimension $\sum_{a = 0}^{n-1} d_a(\cY)$.
\end{proposition}

\begin{proof}
By \cite[Remark 3.13(v)]{Rydh}, the stack $\uHom_k(\cC^\ell_{n,k}, \cX)$ is smooth over $k$. Since $\cZ^\ell_n$ is an open substack of $\uHom_k(\cC^\ell_{n,k}, \cX)$ by \autoref{remarkAboutHomRepSubJetStack}, we have that $\cZ^\ell_n$ is smooth over $k$. Since $\cY^\ell_n$ is a union of connected components of $\cZ^\ell_n$, we have that $\cY^\ell_n$ is smooth over $k$. Thus by \cite[Lemma 3.34]{SatrianoUsatine1}, we only need to show that every fiber of $\sL_1(\cY^\ell_n) \to \cY^\ell_n$ has dimension $\sum_{a = 0}^{n-1} d_a(\cY)$.

Let $\Spec(k') \to \cY^\ell_n$ be a morphism with $k'$ a field extension of $k$. Then $\Spec(k') \to \cY^\ell_n$ corresponds to a map $f: \cC^\ell_{n, k'} \to \cX$ such that the composition $f_0: B\mu_\ell \otimes_k k' = \cC^\ell_{1, k'} \hookrightarrow \cC^\ell_{n, k'} \xrightarrow{f} \cX$ corresponds to a point of $\cY$. By considering the space that parametrizes deformations of the map  $f: \cC^\ell_{n,k'} \to \cX$ to a map $\cC^\ell_{n,k'} \otimes_{k'} k'[x]/(x^2) \to \cX$ and noting that the ideal defining $\cC^\ell_{n,k'}$ in $\cC^\ell_{n,k'} \otimes_{k'} k'[x]/(x^2)$ is isomorphic as a module to $\cO_{\cC^\ell_{n,k'}}$, we see that the fiber of $\sL_1(\cY^\ell_n) \to \cY^\ell_n$ over $\Spec(k') \to \cY^\ell_n$ is an affine space over $k'$ of dimension equal to
\[
	\dim_{k'} H^0(f^*T_\cX).
\]
Since $f^*T_\cX$ is a locally free sheaf on $\cC^\ell_{n,k'}$, it must be isomorphic to a direct sum of twists of $\cO_{\cC^\ell_{n,k'}}$ by \cite[Proposition 3.1]{SatrianoUsatine6}. Thus by definition of the $d_a(\cY)$ and the fact that $f_0$ corresponds to a point of $\cY$, we have that
\[
	f^*T_\cX \cong \bigoplus_{a = 0}^{\ell -1} \cO_{\cC^\ell_{n,k'}}(a)^{d_a(\cY)}.
\]
The proposition thus follows from the fact that
\[
	\dim_{k'} H^0(\cO_{\cC^\ell_{n,k'}}(a)) = \begin{cases}1, & a = 0, \dots, n-1\\ 0, & a= n, \dots, \ell-1 \end{cases},
\]
which is immediate from considering $\cO_{\cC^\ell_{n,k'}}$ as the module $k'[t^{1/\ell}]/(t^{n/\ell})$ with $\mu_\ell$-action.
\end{proof}

\begin{proposition}\label{propSubTruncationFiberComputation}
Let $\ell, n \in \Z_{>0}$, assume $n+1 \leq \ell$, let $\cY$ be a connected component of $I_{\mu_\ell}(\cX)$, let $k'$ be a field extension of $k$, and let $\Spec(k') \to \cY^\ell_n$ be a morphism. Then
\[
	\cY^\ell_{n+1} \times_{\cY^\ell_n} \Spec(k') \cong \bA_{k'}^{d_n(\cY)}
\]
as stacks over $k'$.
\end{proposition}

\begin{proof}
$\Spec(k') \to \cY^\ell_n$ corresponds to a map $f: \cC^\ell_{n, k'} \to \cX$ such that the composition $f_0: B\mu_\ell \otimes_k k' = \cC^\ell_{1, k'} \hookrightarrow \cC^\ell_{n, k'} \xrightarrow{f} \cX$ corresponds to a point of $\cY$. By considering the space that parametrizes deformations of the map  $f: \cC^\ell_{n,k'} \to \cX$ to a map $\cC^\ell_{n+1,k'} \to \cX$ and noting that the ideal defining $\cC^\ell_{n,k'}$ in $\cC^\ell_{n+1,k'}$ is generated by $t^{n/\ell}$ and is thus isomorphic as a module to $\cO_{B\mu_\ell \otimes_k k'}(-n)$, we see that $\cY^\ell_{n+1} \times_{\cY^\ell_n} \Spec(k')$ is isomorphic to an affine space over $k'$ of dimension equal to
\[
	\dim_{k'}H^0((f_0^*T_\cX)(-n)) = d_n(\cY).
\]
\end{proof}

We are now able to prove \autoref{propComparingCohomologyInertiaToZerothJet}.

\begin{proof}[Proof of \autoref{propComparingCohomologyInertiaToZerothJet}]
For each $n = 1, \dots, \ell-1$, the map $\cY^\ell_{n+1} \to \cY^\ell_n$ is smooth and has constant relative dimension by \autoref{miracleFlatnessForStacks}, \autoref{propositionSubJetDimension}, and \autoref{propSubTruncationFiberComputation}. Then along with \autoref{propSubTruncationFiberComputation}, we may apply \autoref{theoremGysinIsIsomorphism} to each $\cY^\ell_{n+1} \to \cY^\ell_n$, and we get
\[
	H^{i+2(\dim\pi_\ell^{-1}(\cY) - \dim\cY)}_c(\pi_\ell^{-1}(\cY))(\dim\pi_\ell^{-1}(\cY) - \dim\cY) \cong H^i_c(\cY).
\]
Thus in order to get the desired result, we just need to show that $\dim\pi_\ell^{-1}(\cY) = \dim\cX$, but this follows from the fact that $\pi_\ell^{-1}(\cY)$ is a union of connected components of $\sJ^\ell_0(\cX)$, which has equidimension $\dim\cX$ by \autoref{theoremSmoothnessOfTruncation}.
\end{proof}

For the remainder of this section, let $\overline{\wt}_\cX: \bigsqcup_{\ell \in \Z_{>0}} |I_{\mu_\ell}(\cX)| \to \Q$ be as defined in \cite[Definition 4.17]{SatrianoUsatine6}, and for any $\ell \in \Z_{>0}$ and connected component $\cY$ of $I_{\mu_\ell}(\cX)$, set
\[
	w(\cY) = \dim\cX - (1/\ell)\sum_{a=1}^{\ell} a d_{-a}(\cY).
\]

\begin{remark}\label{remarkWeightFunctionAndWeightOfComponent}
It is immediate from the definitions that the function $\overline{\wt}_\cX$ takes value $w(\cY)$ on all of $|\cY|$.
\end{remark}

Before we can prove \autoref{theoremStringyIsOrbifold}, we will need the following lemma.

\begin{lemma}\label{lemmaShiftAndWeight}
Let $\cY$ be a connected component of $\bigsqcup_{\ell \in \Z_{>0}}I_{\mu_\ell}(\cX)$. Then
\[
	-\shft_\cX(\cY) = w(\cY) + \dim\cY - \dim\cX.
\]
\end{lemma}

\begin{proof}
Let $\ell \in \Z_{>0}$ be such that $\cY$ is a connected component of $I_{\mu_\ell}(\cX)$. If $k'$ is a field extension of $k$ and $f: B\mu_\ell \otimes_k k' \to \cX$ corresponds to a $k'$-point of $\cY$, then
\[
	f^*T_\cX \cong \bigoplus_{a = 1}^{\ell} \cO_{B\mu_\ell \otimes_k k'} (a)^{d_{a}(\cY)}.
\]
Therefore,
\[
	\shft_\cX(\cY) = \dim\cX - (1/\ell)\sum_{a = 1}^\ell a d_{a}(\cY).
\]
On the other hand,
\begin{align*}
	w(\cY) + \dim\cY - \dim\cX &= \dim\cY - (1/\ell)\sum_{a=1}^\ell ad_{-a}(\cY)\\
	&= d_0(\cY) - (1/\ell)\sum_{a=1}^\ell ad_{-a}(\cY)\\
	&= d_0(\cY) - (1/\ell)\sum_{b=0}^{\ell-1} (\ell - b)d_{b}(\cY)\\
	&= d_0(\cY) - \dim\cX + (1/\ell) \sum_{b=0}^{\ell - 1} b d_b(\cY)\\
	&= -\dim\cX + d_\ell(\cY) + (1/\ell)\sum_{b=1}^{\ell-1} b d_b(\cY)\\
	&= -\dim\cX + (1/\ell)\sum_{b=1}^{\ell} b d_b(\cY) = -\shft_\cX(\cY),
\end{align*}
where the second equality is by \autoref{propositionSubJetDimension}.
\end{proof}

Finally, we may prove \autoref{theoremStringyIsOrbifold}.

\begin{proof}[Proof of \autoref{theoremStringyIsOrbifold}]
Since $\cX$ is tame, $\cA_\cX = |\sJ(\cX)|$, so each $\cY^w_{\cX, n}$ is the open substack of $\sJ_n(\cX)$ supported on the preimage of $\overline{\wt}_\cX^{-1}(w)$ in $\sJ_n(\cX)$. Thus each $\cY^w_{\cX, n+1} \to \cY^w_{\cX, n}$ satisfies the hypotheses of \autoref{theoremGysinIsIsomorphism} by \autoref{theoremSmoothnessOfTruncation} and \cite[Proposition 4.11]{SatrianoUsatine6}, so \autoref{theoremGysinIsIsomorphism} gives that for each $i \in \Z$ and $w \in \Q$,
\[
	H^{i,w}_{\str}(\cX) \cong H^i_c(\cY^w_{\cX, 0}).
\]
For each $\ell \in \Z_{>0}$ and connected component $\cY$ of $I_{\mu_\ell}(\cX)$, let $\cY'$ denote the preimage of $\cY$ in $\sJ_0(\cX)$. By \autoref{remarkWeightFunctionAndWeightOfComponent},
\[
	\cY^w_{\cX, 0} = \bigsqcup_{\cY, w(\cY) = w} \cY',
\]
where $\cY$ varies over all connected components of $\bigsqcup_{\ell \in \Z_{>0}} I_{\mu_\ell}(\cX)$ with $w(\cY) = w$. Therefore for each $i \in \Q$,
\begin{align*}
	H^{\even, i}_{\str}(\cX) &\cong \bigoplus_{w \in \Q, \text{ $i+2w$ is an even integer}} H^{i+2w}_c(\cY^w_{\cX, 0})(w)\\
	&\cong \bigoplus_{w \in \Q, \text{ $i+2w$ is an even integer}} \left( \bigoplus_{\cY, w(\cY) = w} H^{i+2w}_c(\cY')(w) \right)\\
	&\cong \bigoplus_{\cY, \text{ $i+2w(\cY)$ is an even integer}} H^{i+2w(\cY)}_c(\cY')(w(\cY))\\
	&\cong \bigoplus_{\cY, \text{ $i+2w(\cY)$ is an even integer}} H^{i-2(\shft_\cX(\cY))}_c(\cY)(-\shft_\cX(\cY))\\
	&\cong H^{\even, i}_{\orb}(\cX),
\end{align*}
where the second-to-last isomorphism is by \autoref{propComparingCohomologyInertiaToZerothJet} and \autoref{lemmaShiftAndWeight}. Similarly,
\[
	H^{\odd, i}_{\str}(\cX) \cong H^{\odd, i}_{\orb}(\cX).
\]
Finally, this implies
\[
	H^i_{\str}(\cX) \cong H^i_{\orb}(\cX)
\]
by definition and
\[
	\HD_{\str}(\cX) = \HD_{\orb}(\cX)
\]
by definition and \autoref{theoremStringyCohomologyGivesStringyHD}.
\end{proof}

\section{The twisted jet stack of a quotient stack}

We prove \autoref{thm:twisted-jet-quotient-stack} which calculates the twisted jet stacks of quotient stacks. This simultaneously generalizes \cite[Proposition 16]{Yasuda} and \cite[Corollary 3.22]{SatrianoUsatine1}; the former handled quotients by finite groups and the latter handled untwisted jets of quotient stacks.

We first recall \cite[Definition 14]{Yasuda}. If $Y$ is a scheme with $\mu_\ell$-action $\lambda\colon\mu_\ell\times Y\to Y$, then there are two natural action maps $\alpha,\beta\colon\mu_\ell\times\sL_mY\to\sL_mY$:~the first induced by the $\mu_\ell$-action on $D_m=\Spec k[t]/t^{m+1}$ where $t$ has weight $1$, and the second induced by the action of $\mu_\ell$ on $Y$. We have a scheme $\sL^\lambda_mY$ whose $S$-valued points are given by
\[
\sL^\lambda_mY(S)=\{\varphi\colon D_{m,S}\to Y\mid \varphi(\zeta t)=\zeta\cdot_\lambda\varphi(t)\}.
\]
When a group scheme $G$ acts on $Y$ and $\phi\colon\mu_\ell\to G$ is a morphism, we obtain an induced action $\lambda_\phi\colon\mu_\ell\times Y\to Y$ and we let
\[
\sL^\phi_mY=\sL^{\lambda_\phi}_mY.
\]


Note that if $m=(n+1)\ell-1$, then $\cD^\ell_n=[D_m/\mu_\ell]$ since $\Spec k[t^{1/\ell}]/(t^{n+1})\simeq D_m$. The following description was also independently obtained by Yu Shen and Minghao Zhao in upcoming work.

\begin{theorem}\label{thm:twisted-jet-quotient-stack}
Let $n\geq0$, $\ell\geq1$, and let $m=(n+1)\ell-1$, so that $\cD^\ell_n=[D_m/\mu_\ell]$. If $G$ is a smooth affine algebraic group acting on a scheme $X$ over an algebraically closed field $k$ of characteristic $0$, then we have a equivalence of stacks
\[
	\cJ^\ell_n([X/G])\simeq \coprod_{\phi\in(\uHom_{\inj}(\mu_\ell,G)/G)(k)}[\sL^\phi_mX/\sL^\phi_mG],
\]
where $\sL^\phi_mG$ is defined via $\phi$ and the conjugation action of $G$ on itself.
\end{theorem}

We begin with a preliminary result.

\begin{proposition}\label{prop:groupoid}
    Fix $n\geq0$ and $\ell\geq1$, and let $m=(n+1)\ell-1$, so that $\cD^\ell_n=[D_m/\mu_\ell]$. If $G$ is a smooth affine algebraic group acting on a scheme $X$, then $\cJ^\ell_n([X/G])$ is equivalent to the stack whose $S$-valued points is the category of triples
    \[(\pi:P\to D_{m,S},\ \tau\colon P\to X,\ \lambda\colon\mu_\ell\times P\to P)\]
    where $\lambda$ defines a $\mu_\ell$-action such that
    \begin{enumerate}
        \item $\pi$ is a $G$-torsor and is $\mu_\ell$-equivariant.
        \item $\tau$ is $G$-equivariant,
        \item $\tau$ is $\mu_\ell$-equivariant where $X$ has the trivial $\mu_\ell$-action,
        \item the $\mu_\ell$-action on $P$ induced by $\lambda$ is free,
        \item the $\mu_\ell$-actions and $G$-actions on $P$ commute.
    \end{enumerate}
Morphisms between two triples $(\pi,\tau,\lambda)$ and $(\pi',\tau',\lambda')$ are given by $G$-torsor isomorphisms $P\to P'$ compatible with each of the three pairs of morphisms.
\end{proposition}

\begin{remark}
One could equivalently formulate \autoref{prop:groupoid} by requiring that $P$ and $X$ have right $G$-actions, and that $P$ and $D_{m,S}$ have left $\mu_\ell$-actions. Then commutativity of the $G$ and $\mu_\ell$-actions on $P$ is natural. We chose the above formulation since we prefer to keep left actions on both $X$ and $D_{m,S}$.
\end{remark}

\begin{proof}[{Proof of \autoref{prop:groupoid}}]
Representable maps $\cD^\ell_{n,S} \to [X/G]$ are equivalent to the groupoid consisting of a $G$-torsor $Q\to\cD^\ell_{n,S}$ and $G$-equivariant map $Q\to X$, where $Q$ is an algebraic space. Consider the cartesian diagram
\[
\xymatrix{
P\ar[r]\ar[d]_-{\pi} & Q\ar[r]\ar[d] & X\ar[d]\\
D_{m,S}\ar[r] & \cD^\ell_{n,S}\ar[r] & [X/G]
}
\]
We have an induced $\mu_\ell$-action on $P$ making $\pi$ a $\mu_\ell$-equivariant morphism; the $\mu_\ell$-action on $P$ is free since $Q$ is an algebraic space. We see $\pi$ is a $G$-torsor and $\tau\colon P\to X$ is $G$-equivariant. Since $\mu_\ell$ acts on $P$ via $G$-torsor automorphisms, the $\mu_\ell$-actions and $G$-actions commute. Furthermore, the groupoid consisting of $\pi$, $\tau$, and a $\mu_\ell$-action on $P$ satisfying the hypotheses of the proposition is equivalent to the groupoid consisting of a $G$-torsor $Q\to\cD^\ell_{n,S}$ and $G$-equivariant map $Q\to X$, where $Q$ is an algebraic space. Indeed, the fact that $\pi$ is $\mu_\ell$-equivariant combined with commutativity of the $\mu_\ell$-actions and $G$-actions on $P$ means that $\mu_\ell$ acts on $P$ via $G$-torsor automorphisms over the action on $D_{m,S}$; as a result, $\pi$ descends to a $G$-torsor $Q\to\cD^\ell_{n,S}$, and freeness of the $\mu_\ell$-action on $P$ ensures $Q$ is an algebraic space. Then $G$-equivariance and $\mu_\ell$-invariance of $\tau$ implies that $P\to X$ descends to a $G$-equivariant map $Q\to X$.
\end{proof}

\begin{lemma}\label{lem:Hom-mod-conj}
Let $G$ be a smooth affine algebraic group scheme over an algebraically closed field $k$ of characteristic $0$, and consider the conjugation action of $G$ on $\uHom(\mu_\ell,G)$. Then $\uHom(\mu_\ell,G)/G$ is a finite disjoint union of copies of $\Spec k$.
\end{lemma}
\begin{proof}
By Theorem 2 and Proposition 4.5 of \cite{BrionHomAlg}, $\cH:=\uHom(\mu_\ell,G)$ is represented by a smooth affine finite type $k$-scheme, so $\cH/G=\Spec H^0(\cO_\cH)^G$. Furthermore, Theorem 2 of (loc.~cit.) shows
\[
\Psi\colon G\times\cH \to \cH\times\cH, \quad (g,\phi) \mapsto (g\phi g^{-1},\phi)
\]
is smooth (hence universally open). As a result, pulling back $\Psi$ by $\cH\xrightarrow{(\id,\phi_0)}\cH\times\cH$ for any $k$-point $\phi_0$, we see all $k$-points of $\cH$ have open orbits. As a result, the complement of any orbit is also the union of open orbits, hence open, i.e., every orbit is both open and closed. Since $\cH$ is finite type, it is therefore a finite disjoint union of orbits $U_i$. Then $\cH/G=\coprod_i U_i/G$. Since $G$ and $U_i$ are smooth, we see the homogeneous variety $U_i/G$ is smooth as well; on the other hand, $U_i/G$ is topologically a point, hence $U_i/G=\Spec k$.
\end{proof}

The following technical result shows that $\mu_\ell$-actions on $G$-torsors over $D_{m,S}$ take a very specific form.

\begin{lemma}\label{lem:std-form}
Let $G$ be a smooth affine algebraic group over an algebraically closed field $k$, let $S$ be a locally finitely presented $k$-scheme, and let $\pi\colon P\to D_{m,S}$ be a left $G$-torsor. Suppose we have a $\mu_\ell$-action on $P$ commuting with the $G$-action such that $\pi$ is $\mu_\ell$-equivariant for the standard $\mu_\ell$-action on $D_{m,S}$. Then after possibly replacing $S$ by an \'etale cover, there exists a trivialization $P\simeq D_{m,S}\times G$ of the $G$-torsor and a morphism $\phi\colon\mu_\ell\to G$ such that the induced $\mu_\ell$-action on $D_{m,S}\times G$ is given by
\[
\zeta\cdot(z,g)=(\zeta z,g\phi(\zeta)^{-1}).
\]
Furthermore, if the $\mu_\ell$-action on $P$ is faithful, then we can choose $\phi$ to be injective.
\end{lemma}
\begin{proof}
%
Let $P_r:=P\times_{D_{m,S}}D_{r,S}$ and let $\pi_r\colon P_r\to D_{r,S}$ be the induced $G$-torsor. Since $P_0\to S$ is smooth, it has a section \'etale locally. Replacing $S$ by an \'etale cover, we may assume $S$ is affine and that we have a section $s_0\colon S\to P_0$. By induction, we may assume we have lifted this to a section $s_{m-1}\colon D_{m-1,S}\to P_{m-1}$ of $\pi_{m-1}$. Then the obstruction to deforming this section lives in $H^1(s_{m-1}^*T_{P_{m-1}/D_{m-1,S}},(t^m)/(t^{m+1}))$ which vanishes. 

It will be notationally convenient at this point to work with right $G$-torsors instead of left $G$-torsors. To do so, consider the scheme automorphism of $D_{m,S}\times G$ given by $\Psi(z,g)=(z,g^{-1})$. Then the action $\zeta\cdot(z,g)=(\zeta z,g\phi(\zeta)^{-1})$ in the statement of the result becomes $\Psi^{-1}(\zeta\cdot\Psi(z,g))=(\zeta z,\phi(\zeta)g)$. Hence, if we work from now on with the right $G$-torsor $P=P_m=D_{m,S}\times G$, our new goal is to show $\zeta\cdot(z,g)=(\zeta z,\phi(\zeta)g)$.

Our trivializations $s_r$ induce compatible isomorphisms $P_r\simeq D_{r,S}\times G$ over $D_{r,S}$. The $\mu_\ell$-action on $P_0\simeq G_S$ is determined by where it sends the identity section, hence is given by a map $\phi_S\colon\mu_{\ell,S}\to G_S$. Since the $\mu_\ell$-action commutes with the $G$-action (i.e., it acts as $G$-torsor automorphisms), it is given by left multiplication, hence $\zeta\cdot g=\phi_S(\zeta)g$. We see $\phi_S$ is a homomorphism since $\zeta_1\cdot\zeta_2\cdot 1=\zeta_1\zeta_2\cdot 1$ and that if the $\mu_\ell$-action is faithful then $\phi_S$ is injective. Furthermore, changing our choice of trivialization changes $\phi_S$ by the conjugation action of $G$. We therefore have a map
\[
S\to \uHom(\mu_\ell,G)/G.
\]

By \autoref{lem:Hom-mod-conj}, $\uHom(\mu_\ell,G)/G$ is finite disjoint union of copies of $\Spec k$. Replacing $S$ by its connected components, we may assume there is a morphism $\phi\colon\mu_\ell\to G$ defined over $k$ such that $\phi\times_k S$ and $\phi_S$ define the same $S$-point of $\uHom(\mu_\ell,G)/G$. Since $G$ is smooth, after replacing $S$ by an \'etale cover and possibly choosing a different trivialization of $P_0$, we may assume $\phi_S=\phi\times_k S$; note that since $S$ is locally finitely presented over $k$, if $\phi_S$ is injective then $\phi$ is as well since $S\to\Spec k$ is an fppf cover.

Next, we assume by induction there is a trivialization $P_{m-1}\simeq D_{m-1,S}\times G$ such that the induced $\mu_\ell$-action on $D_{m-1,S}\times G$ is given by $\zeta\cdot(z,g)=(\zeta z,\phi(\zeta)g)$. As we have shown (in the first paragraph of this proof), every trivialization of $P_{m-1}$ lifts to a trivialization of $P_m$. Now, the $\mu_\ell$-action on $P_m$ with this choice of trivialization is determined by where it sends the identity section $\mu_\ell\times D_{m,S}\to D_{m,S}\times G_S$; we see that for all scheme-theoretic points $\zeta_1,\zeta_2$ of $\mu_\ell$ and $z$ of $D_{m,S}$, we have
\begin{align*}
(\zeta_1\zeta_2 z, \beta(\zeta_1\zeta_2,z)) &=\zeta_1\zeta_2\cdot(z,1)=\zeta_1\cdot\zeta_2\cdot(z,1)\\
&=\zeta_1\cdot(\zeta_2 z,\beta(\zeta_2,z))=\zeta_1\cdot(\zeta_2 z,1)\cdot\beta(\zeta_2,z)\\
&=(\zeta_1\zeta_2 z,\beta(\zeta_1,\zeta_2 z))\cdot\beta(\zeta_2,z)=(\zeta_1\zeta_2 z,\beta(\zeta_1,\zeta_2 z)\beta(\zeta_2,z)).
\end{align*}
So letting
\[
f\colon\mu_\ell\times D_{m,S}\to G
\]
given by
\[
f(\zeta,z)=\beta(\zeta,z)\phi(\zeta)^{-1},
\]
we see
\[
f(\zeta_1\zeta_2,z)=f(\zeta_1,\zeta_2 z)\phi(\zeta_1)f(\zeta_2,z)\phi(\zeta_1)^{-1}.
\]
Letting
\[
f'(\zeta,z)=f(\zeta,\zeta^{-1}z),
\]
we then have
\[
f'(\zeta_1\zeta_2,z)=f'(\zeta_1,z)\phi(\zeta_1)f'(\zeta_2,\zeta_1^{-1}z)\phi(\zeta_1)^{-1}.
\]
Note that $f'\colon\mu_\ell\times D_{m,S}\to G$ may be viewed as a map $\mu_{\ell,S}\to\sL_mG$, and by construction, it factors as 
\[
F'\colon\mu_{\ell,S}\to\cK=\ker(\sL_mG\to\sL_{m-1}G).
\]
Thus, if we $\mu_\ell$ act on $\cK$ as $\zeta\star\varphi(z)=\phi(\zeta)\varphi(\zeta^{-1}z)\phi(\zeta)^{-1}$, then we see
\[
F'(\zeta_1\zeta_2)=F'(\zeta_1)\cdot (\zeta_1\star F'(\zeta_2)),
\]
i.e., $F'$ is a $1$-cocycle.

By deformation theory, we know $\cK=\mathbb{V}(\g^\vee\otimes(t^m)/(t^{m+1}))$ is the vector group associated to the vector space $\g^\vee\otimes(t^m)/(t^{m+1})$, where $\g$ is the Lie algebra of $G$. Note that the action of $\mu_\ell$ on $\cK$ is in fact a linear action, i.e., it is induced by an action on $\g^\vee\otimes(t^m)/(t^{m+1})$ since this action is comprised of two separate actions:~the first is where $\mu_\ell$ acts via $\phi$ and the adjoint action of $G$; the second is where $\mu_\ell$ acts on jets by $\varphi(t)\mapsto\varphi(\zeta t)$ which is given by the coaction map 
\[
\g^\vee \otimes (t^m)/(t^{m+1}) \to \g^\vee \otimes (t^m)/(t^{m+1})\otimes k[u]/(u^\ell-1)
\]
which is the identity on $\g^\vee$ and $t^m \otimes 1 \mapsto t^m \otimes u^m$ on the second tensor factor.

Since $\mu_\ell$ is linearly reductive, we know $H^1(\mu_\ell,\g^\vee \otimes (t^m)/(t^{m+1}))=0$, and so $F'$ is in fact a $1$-coboundary, i.e., 
there is a $k$-point of $h\in\cK(k)\subset(\sL_nG)(k)$ such that $F'(\zeta)=h/(\zeta\star h)$. Letting $g\colon D_{m,S}\to G$ denote the jet corresponding to $h$, we see then that
\[
f'(\zeta,z)=g(z)\phi(\zeta)g(\zeta^{-1}z)^{-1}\phi(\zeta)^{-1}.
\]
Hence,
\[
\beta(\zeta,z)=f'(\zeta,\zeta z)\phi(\zeta)=g(\zeta z)\phi(\zeta)g(z)^{-1}.
\]
As a result, changing our trivialization of $P_m$ by $g(z)$, we may assume $\beta(\zeta,z)=\phi(\zeta)$, as desired.
\end{proof}

\begin{proof}[{Proof of \autoref{thm:twisted-jet-quotient-stack}}]
For each injective morphism $\phi\colon\mu_\ell\to G$, we define a map
\[
\alpha_\phi\colon[\sL^\phi_mX/\sL^\phi_mG]\to \cJ^\ell_n([X/G]).
\]
Let $(\eta,f)\in[\sL^\phi_mX/\sL^\phi_mG](S)$ with $\eta\colon \cP\to S$ a left $\sL^\phi_mG$-torsor and $f\colon \cP\to\sL^\phi_mX$ a $\sL^\phi_mG$-equivariant map. By definition, $f$ corresponds to a map
\[
F\colon \cP\times D_{m,S}\to X
\]
 such that
 \[
 F(p,\zeta z)=\phi(\zeta)F(p,z)\quad\textrm{and}\quad F(\gamma p,z)=\ev(\gamma,z) F(p,z)
 \]
 for all scheme-theoretic points $\zeta\in\mu_\ell(S')$, $z\in D_{m,S}(S')$, and $\gamma\in(\sL^\phi_mG)(S')$, where $\ev(\gamma,z)\in G(S')$ denotes the evaluation of $\gamma\colon D_{m,S'}\to G$ at $z$. Let
 \[
 P=(G\times\cP\times D_{m,S})/\sL^\phi_mG
 \]
where $\sL^\phi_mG$ acts via the free right action
\[
(g,p,z)\cdot\gamma=(g\ev(\gamma,z),\gamma^{-1} p,z).
\]
We denote by $[g,p,z]$ the image of a scheme-theoretic point $(g,p,z)$ of $G\times\cP\times D_{m,S}$. Then
\[
\pi\colon P\to D_{m,S},\quad\pi([g,p,z])=z
\]
is a left $G$-torsor and we have a $G$-equivariant map
\[
\tau\colon P\to X,\quad\tau([g,p,z])=g F(p,z).
\]
We define a $\mu_\ell$-action on $P$ given by 
\[
\zeta\cdot[g,p,z]=[g\phi(\zeta)^{-1},p,\zeta z].
\]
Then we see $\pi$ is $\mu_\ell$-equivariant, $\tau$ is $\mu_\ell$-invariant, and the $G$-action and $\mu_\ell$-action on $P$ commute. It is also straightforward to see that the $\mu_\ell$-action on $P$ is free. This construction is functorial hence extends to morphisms as well. Thus, by \autoref{prop:groupoid}, we have defined our map $\alpha_\phi$.

By \autoref{lem:Hom-mod-conj}, $\uHom_{\inj}(\mu_\ell,G)/G$ is a finite disjoint union of copies of $\Spec k$. Thus, both stacks in the statement of our theorem are locally finitely presented over $k$. As a result, to prove 
\[
\alpha=\coprod_{\phi\in(\uHom_{\inj}(\mu_\ell,G)/G)(k)} \alpha_\phi
\]
is an equivalence, it suffices to prove full faithfulness and essential surjectivity on the level of $S$-points for all locally finitely presented $k$-schemes $S$. 

To prove $\alpha$ is faithful, we must show that if $\Gamma$ is an automorphism of $(\eta,f)\in[\sL^\phi_mX/\sL^\phi_mG](S)$ and $\alpha_\phi(\Gamma)=\id$, then $\Gamma=\id$. We may look fppf locally where $\eta\colon\cP\to S$ is the trivial torsor.
 Fixing a trivialization, we may assume $\cP=\sL^\phi_mG \times_k S$; note this implies our $G$-torsor $\pi\colon P=(G\times\cP\times D_{m,S})/\sL^\phi_mG\to D_{m,S}$ in the pair $(\pi,\tau)=\alpha_\phi(\eta,f)$ is also trivial since $[g,p,z]=[gp,1,z]$, showing that every point $P$ can be written uniquely in the form $[g,1,z]$. Now $\Gamma$ is given by right multiplication by some $\gamma\in(\sL^\phi_mG)(S)\subset(\sL_mG)(S)$ and by construction, $\alpha_\phi(\Gamma)$ is the automorphism of $P=(G\times\cP\times D_{m,S})/\sL^\phi_mG$ given by
\[
[g,1,z]\mapsto [g,\gamma,z]=[g\ev(\gamma,z),1,z].
\]
In other words, it is left multiplication by $\ev(\gamma,-)\in G(D_{m,S})$. If $\ev(\gamma,-)=\id$, then for all $S'$-points $z'\in D_{m,S}(S')$, we have $\id=\ev(\gamma,z')=\gamma(z')$, i.e., $\gamma\colon D_{m,S}\to G$ is the constant jet $\gamma=\id$.

We now prove $\alpha$ is full. Let $(\eta,f)\in[\sL^\phi_mX/\sL^\phi_mG](S)$ and $(\eta',f')\in[\sL^{\phi'}_mX/\sL^{\phi'}_mG](S)$. Let $(\pi,\tau)=\alpha(\eta,f)$, $(\pi',\tau')=\alpha(\eta',f')$, and $\xi\colon(\pi,\tau)\to(\pi',\tau')$ be an isomorphism. We must show $\xi=\alpha(\Gamma)$ for some $\Gamma\colon(\eta,f)\to(\eta',f')$. Since $\alpha$ is faithful, it suffices to look locally where we may assume $\eta\colon\cP=\sL^\phi_mG \times_k S\to S$ and $\eta'\colon\cP'=\sL^{\phi'}_mG \times_k S\to S$ are the trivial torsors, hence $\pi$ and $\pi'$ are also trivial torsors. Our first goal is to prove $\phi$ and $\phi'$ are conjugate. Since $\pi$ and $\pi'$ are trivial, $\xi$ is given by right multiplication by some $g\in G(D_{m,S})$, i.e., there exists $\gamma\colon D_{m,S}\to G$ and $g=\ev(\gamma,-)$. Furthermore, by \autoref{prop:groupoid}, right multiplication by $g$ must commute with the $\mu_\ell$-action. Thus, we see
\[
[\ev(\gamma,z)\phi'(\zeta)^{-1},1,\zeta z]=\zeta\cdot\xi([1,1,z])=\xi(\zeta\cdot[1,1,z])=[\phi(\zeta)^{-1}\ev(\gamma,\zeta z),1,\zeta z],
\]
which implies
\[
\phi'(\zeta)=\ev(\gamma,\zeta z)^{-1} \phi(\zeta) \ev(\gamma,z).
\]
In particular, since $\phi$ and $\phi'$ are defined over $k$, pulling back by any map $x\colon\Spec k\to S\to D_{m,S}$ and using that $x$ is a fixed point under the $\mu_\ell$-action, we see $\phi' = \ev(\gamma,x)^{-1} \phi \ev(\gamma,x)$, i.e., $\phi$ and $\phi'$ are conjugate. Thus, we may now assume $\phi'=\phi$. Hence, the above equation shows
\[
\ev(\gamma,\zeta z)=\phi(\zeta)\ev(\gamma,z)\phi(\zeta)^{-1},
\]
i.e., $\gamma\in(\sL^\phi_mG)(S)$. Lastly, let $F\colon\cP\times D_{m,S}\to X$ and $F'\colon\cP\times D_{m,S}\to X$ denote the maps corresponding to $f$ and $f'$. Then examining how $\xi$ behaves with respect to $\tau$ and $\tau'$, we see
\begin{align*}
F(1,z)=\tau([1,1,z])&=\tau'([\ev(\gamma,z),1,z])\\
&=\ev(\gamma,z)\tau'([1,1,z])=\ev(\gamma,z)F'(1,z)=F'(\gamma,z).
\end{align*}
Thus, we see that right multiplication by $\gamma$ defines an isomorphism $\Gamma\colon(\eta,f)\to(\eta',f')$ such that $\alpha(\Gamma)=\xi$.

Finally, we prove $\alpha$ is essentially surjective. Given an object $\cJ^\ell_n([X/G])(S)$, it is given by a $G$-torsor $\pi\colon P\to D_{m,S}$, a $G$-equivariant map $\tau\colon P\to X$, and a $\mu_\ell$-action $\mu_\ell\times P\to P$ satisfying the hypotheses of \autoref{prop:groupoid}. Since $\alpha$ is fully faithful, we may look locally on $S$, in which case, by \autoref{lem:std-form}, we may assume $P=G\times D_{m,S}$ is the trivial torsor and the $\mu_\ell$-action is given by $\zeta\cdot(g,z)=(g\phi(\zeta)^{-1},\zeta z)$. We have a $G$-equivariant map $\tau\colon G\times D_{m,S}\to X$ satisfying the hypotheses of \autoref{prop:groupoid}. Since $\tau$ is $\mu_\ell$-invariant and $G$-equivariant, if $\varphi(z):=\tau(1,z)$, then
\[
\varphi(z)=\tau(1,z)=\tau(\zeta\cdot(1,z))=\tau(\phi(\zeta)^{-1},\zeta z)=\phi(\zeta)^{-1}\tau(1,z)=\phi(\zeta)^{-1}\varphi(z)
\]
i.e., $\varphi$ defines an $S$-point of $\sL^\phi_mX$. In particular, if $\eta\colon\cP=\sL^\phi_mG\times D_{m,S}\to D_{m,S}$ is the trivial $\sL^\phi_mG$-torsor and $f\colon\cP\to\sL^\phi_mX$ is the $\sL^\phi_mG$-equivariant map $f(\gamma,z)=\gamma\varphi(z)$, then $(\eta,f)\in[\sL^\phi_mX/\sL^\phi_mG](S)$ and $\alpha(\eta,f)=(\pi,\tau)$.
\end{proof}

\section{Example:~the threefold $A_1$-singularity}
In this section, we illustrate the computational feasibility of stringy cohomology. Consider the singularity
$Y=\Spec \CC[y_1,y_2,y_3,y_4]/(y_1y_2-y_3y_4)$,
which admits a small resolution by the Artin stack
$\X=[\AA^4_{x_1,x_2,x_3,x_4}/\GG_m]$;
here $\GG_m$ acts with weights $(1,1,-1,-1)$ and $\pi:\X\to Y$ is induced by the map $\bA^4\to Y$ given by
\[(x_1,x_2,x_3,x_4)\mapsto (x_1x_3,x_2x_4,x_1x_4,x_2x_3).\] 
In the following result, we compute the stringy cohomology of $\cX$.

\begin{proposition}\label{prop:HstrA1-3d}
For each $j\in\ZZ$, we have
\[
H_{\mathrm{str}}^j(\mathcal{X}) \cong
\begin{cases}
\mathbb{Q}(2i),
& j=1-4i \text{ for some } i\geq 0,\\
\mathbb{Q}(2i-1),
& j=3-4i \text{ for some } i\geq 0,\\
\mathbb{Q}(2i-2),
& j=4-4i \text{ for some } i\geq 0,\\
\mathbb{Q}(2i-3),
& j=6-4i \text{ for some } i\geq 0,\\
0,
& \text{otherwise}.
\end{cases}
\]
\end{proposition}

Let $n\in\ZZ_{\geq0}$, $\ell\in\ZZ_{>0}$, and $m=(n+1)\ell-1$. \autoref{thm:twisted-jet-quotient-stack} states
\[\J^\ell_n[\AA^4/\GG_m]\cong \coprod_{[\phi]\in\left(\Hom_{\inj}(\mu_\ell,\GG_m)/\GG_m\right)(\CC)}[\L_m^\phi\AA^4/\L_m^\phi\GG_m]\]
Choose $[\phi]\in (\Hom_{\inj}(\mu_\ell,\GG_m)/\GG_m)(\CC)$ and let $\phi$ be any representative of the class $[\phi]$. A field-valued point $\L_m^{\phi}\AA^4(k')$ is equivalent to a morphism on the ring
\[k'[x_1,x_2,x_3,x_4]\to  k'[t]/(t^{m+1}),\qquad x_i\mapsto x_i(t)\] 
given by four polynomials $x_i(t)\in k'[t]/t^{m+1}$, such that they are $\phi$-equivariant
\[x_i(\zeta t)=\begin{cases}
    \phi(\zeta)x_i(t), &i=1,2\\
    \phi(\zeta)^{-1}x_i(t), &i=3,4.
\end{cases}\]
Two such choices yield isomorphic points of $\J_n^\ell\X(k')$ if they differ by the action of an element $g\in \L_m^{\phi}\GG_m(k')$. Note that $\L_m^{\phi}\GG_m=\L_n\GG_m$ because $\GG_m$ is abelian. Hence the jets given by $x_i(t)$ and $x_i'(t)$ are isomorphic if 
\[x_i'(t)=\begin{cases}
    g(t)x_i(t), &i=1,2\\
    g(t)^{-1}x_i(t), &i=3,4.
\end{cases}\]

Let us fix a jet $\psi\in\J_n^\ell\X(k')$. Since $\sL^\phi\bG_m$ is special, $\psi$ is induced from a choice of morphism $\phi:\zeta\mapsto \zeta^a$ and polynomials $x_i(t)$ with $\gcd(a,\ell)=1$ 
such that
\[ x_i(t)= 
\begin{cases} 
t^{a}z_i(t^\ell), & i=1,2\\ 
t^{\ell-a}z_i(t^\ell), & i=3,4, 
\end{cases} 
\] 
for some $z_i(t)\in k'[t]/(t^{n+1})$. 
In particular, $\val_t(x_i)$ belongs to $a+\ell\ZZ$ for $i=1,2$ and $-a+\ell\ZZ$ for $i=3,4$.

The Edidin-Rydh reduction $\X'$ with respect to the good moduli space morphism $\pi:\X\to Y$ is given by $\X'=\Bl_0\X\setminus \cS'$ where $\cS'$ the strict transform of the locus $[\cS/\bG_m]$, where 
\[
\cS=\pi^{-1}\pi(0)=\{x_1=x_2=0\}\cup\{x_3=x_4=0\}\seq \AA^4.
\]
The image of the induced map $|\J(\X')|\to |\J(\X)|$ is denoted $\A_\X$. Let $\A_n=\theta_n(\A_n)$ and $\A_n^{w}=\theta_n(\A_\X^{w})$. By \autoref{thm:twisted-jet-quotient-stack}, to determine the topological spaces $\A_n$ (respectively $\A_n^{w}$), it is sufficient to determine which field-valued points of $\coprod_{[\phi]\in\left(\Hom_{\inj}(\mu_\ell,\GG_m)/\GG_m\right)(\CC)}|[\L_m^\phi\AA^4/\L_m^\phi\GG_m]|$ are contained in $\A_n$ (respectively $\A_n^{w}$).

\begin{lemma}
    For every $n\geq 1$, the points of $\A_n$ are given by the image of polynomials $x_i(t)\in k'[t]/(t^{m+1})$ such that
    \[m_1:=\min(\val_t(x_1),\val_t(x_2)) = \min(\val_t(x_3),\val_t(x_4))=:m_2.\] 
    Furthermore, $\A_n$ consists of image of $1$-twisted or $2$-twisted arcs.
\end{lemma}
\begin{proof}
    By formal smoothness of $\X'$, we know
    \[\A_n=\theta_n(\im(|\J^\ell(\X')|\to|\J^\ell(\X)|))=\im(|\J_n^\ell(\X')|\to|\J_n^\ell(\X)|).\]
    Suppose a twisted jet is given by $\phi:\zeta\mapsto\zeta^a$ and polynomials $x_i(t)\in k'[t]/(t^{m+1})$. Without loss of generality, we assume this twisted jet is centered in the chart $U_1=\{v_1\neq0\}$ of 
    \[\Bl_0\AA^4\seq \AA^4_{x_1,x_2,x_3,x_4}\times\PP^3_{v_1:v_2:v_3:v_4}\] 
    with coordinates $(x_1,v_2,v_3,v_4)$. 
    
    Note that $U_1$ is $\GG_m$-invariant, where the induced $\GG_m$-action has weights $(1,0,-2,-2)$. Removing the preimage $\tilde S$ of $\cS'$ gives $U_1\setminus \tilde S=\{v_3\neq 0\}\cup\{v_4\neq 0\}$. Then $\phi$-equivariance gives $v_3(\zeta t)=\zeta^{-2a}v_3(t)$ and $v_4(\zeta t)=\zeta^{-2a}v_4(t)$. Since $v_3(0)\neq 0$ or $v_4(0)\neq 0$ and $\gcd(\ell,a)=1$, we conclude $\ell=1$ or $\ell=2$.

    Note that $x_2(t)=x_1(t)v_2(t)$, so $\val_t(x_2)\geq \val_t(x_1)$ and $m_1=\val_t(x_1)$. On the other hand, if $x_1(t)=0$, then $x_i(t)=x_1(t)v_i(t)=0$ for $i=3,4$. If $x_1(t)\neq 0$, then $v_i(t)=x_i(t)/x_1(t)$ for $i=3,4$ and one of $v_3(0),v_4(0)$ is non-zero, meaning $m_2=\min(\val_t(x_3),\val_t(x_4))=\val_t(x_1)=m_1$. 

    Conversely, any twisted jet of $\J_n(\X)$ that is the image of polynomials \[x_1(t),x_2(t),x_3(t),x_4(t) \qquad \text{such that}\qquad  m_1=\val_t(x_1)=m_2\] and $x_1(t)\neq 0$ lifts to the jet which is the image of $(x_1(t),v_2(t),v_3(t),v_4(t))$ given by $v_i(t)=x_i(t)/x_1(t)$ for $i=2,3,4$ in the chart $U_1$.
\end{proof}

 Recall that the weight $w$ component of $\A_\X$ is denoted $\A_\X^w=\A_\X\cap \wt_\X^{-1}(w)$.

\begin{lemma}\label{lem:weight-2-twist}
    For any $2$-twisted arc $\psi\in \A_\X$, we have $\wt(\psi)=2$. Hence,
    \[\A_\X^0=\A_\X\cap|\J_n^1(\X)|,\qquad \A_\X^2=\A_\X\cap |\J_n^2(\X)|.\]
\end{lemma}
\begin{proof}
The weight function factors through the truncation 
$|\J^2(\X)|\to |I_{\mu_2}(\X)|$.
So it is enough to compute it for the morphism 
$\f:B\mu_2\to \X=[\AA^4/\GG_m]$ centered at the origin given by an embedding of automorphism groups $\phi:\mu_2\to \GG_m, \zeta\mapsto \zeta^a$ with $\gcd(a,\ell)=1$.

The complex $\f^*L_\X$ is the quasi-isomorphic to the complex of $\mu_2$-representations 
\[\left[\Omega^1_{\bA^4,0}\longrightarrow\fg^\vee\right],\] 
where $\fg$ is the Lie algebra of $\bG_m=\Spec k[u^\pm]$ and the differential given by the evaluation at $0$ of the map $dx_i\mapsto x_i du$ for $i=1,2$ and $dx_i\mapsto -x_i du$ for $i=3,4$; in other words, the map vanishes.
Hence,
\[L^0\varphi^*L_\X\cong \Omega^1_{\AA^4,0}\cong\cO_{B\mu_2}(1)^{\oplus 4}.\] 
On the other hand, since $\GG_m$ is abelian, the coadjoint action on $\fg^\vee$ is trivial, so
\[L^1\varphi^*L_\X\cong\cO_{B\mu_2}(2).\] 
We conclude using the formula of \cite[Remark 4.18]{SatrianoUsatine6} that
\[\wt_\X(\gamma)=3+\frac12(2-4)=2.\] 
The final assertion in the statement of the result follows because all untwisted arcs have weight $0$.
\end{proof}

Let $\cC_{n,i}^{w}$ denote the locus of $\A_n^{w}$ the image of $\sL^\phi_m(\bA^4)$ where
\[m_1=m_2=i,\qquad\text{for } i\in [0,m]\cup\{\infty\}.\]
In particular, $C_{n,\infty}^{0}$ is the singleton of the untwisted arc corresponding to $x_i(t)=0$, and $\cC_{n,i}^{2}$ is only non-empty for odd $i$.
Note that $\overline{\A_n^{w}}=|\J_n\X|\cap \wt_{\X}^{-1}(w)$,
and $\overline{\A_n^{w}}\setminus \A_n^{w}$ consists of the jets corresponding to polynomials such that $m_1\neq m_2$. So, 
\[\left|\Y_0^{\A_n^{w}}\right|=\overline{\A_n^{w}}\setminus (\overline{\overline{\A_n^{w}}\setminus \A_n^{w}})=\begin{cases}
    \cC^0_{n,0},&w=0\\
    \cC^2_{n,1},&w=2;
\end{cases}\]
recall that $\Y^\cC$ denotes the reduced substack associated to a locally closed subset of points $\cC$, and $\Y^\cC_i$ denotes the $i$-th component of its canonical decomposition. 
Repeating the above argument shows that 
\[ \left|\Y_i^{\A_n^{w}}\right|=\begin{cases}
    \cC^0_{n,i},&w=0\\
    \cC^2_{n,2i+1},&w=2
\end{cases} \qquad \text{ for $i=0,1,\dots, n$,}\]  
and the last component in the canonical decomposition of $\Y^{\A^{0}_n}$ is $|\Y_{n+1}^{\A_n^{0}}|=\cC_{n,\infty}^1$. This completely describes the canonical decompositions of $\Y^{\A_n^w}$ for $w=0,2$.

\begin{lemma}\label{lem:stab} 
We have the following descriptions of the components \(\Y^{\cC_{n,i}^w}\). For each $0\leq i\leq (n+1)/2$, 
 \[ \Y^{\cC_{n,i}^0}\cong \Y^{\cC_{n,2i+1}^2}\cong 
     \L_{n-i}M\times B\GG_a^i,
 \]
 where $M$ is the locus of $2\times2$ matrices of rank 1, and 
 \[
 \e(\cC_{n,i}^0)=\e(\cC_{n,2i+1}^2)=(\LL+1)(\LL^2-1)\LL^{3n-4i}. 
 \] 
 Furthermore, 
 \[ \Y^{\cC_{n,\infty}^0}=B\L_m\GG_m, \qquad \e(\cC_{n,\infty}^0)=(\LL-1)^{-1}\LL^{-n}. \] 

\end{lemma} 
\begin{proof} 
For $i=0,\dots,n$, let
\[C_{n,i}^0:=\{(f(t),g(t))\mid \min(\val_t(f),\val_t(g))=i\}.\] 
The locus 
$\Y^{\cC_{n,i}^0}\subseteq[\L_n\AA^4/\sL_n\GG_m]$ is equal to the quotient of $C_{n,i}^0\times C_{n,i}^0\seq \L_n\AA^4$ by $\L_n\GG_m$. A direct deformation theory computation shows that for $i\leq (n+1)/2$, 
\[\ker\left(\L_n\GG_m\to\L_{n-i}\GG_m\right)\cong\GG_a^i.\]
The geometric points of this subgroup are polynomials of form $1+t^{n+1-i}k'[t]/(t^{n+1})$ from which we see this $\bG_a^i$ acts trivially on $C_{n,i}^0\times C_{n,i}^0$. On the other hand, consider the action of the quotient $\L_{n-i}\GG_m$ on $C_{n,i}^0\times C_{n,i}^0$. Writing a point of $C_{n,i}^0\times C_{n,i}^0$ as
\[(f_1(t),g_1(t);f_2(t),g_2(t))\]
where $\L_n\GG_m$ acts with weight $1$ on the first two entries, and $-1$ on the last two, the four invariant functions $f_if_{3-i}$ and $f_ig_{3-i}$ form a matrix
\[A(t)=\begin{pmatrix}
    A_1&A_2\\A_3&A_4
\end{pmatrix}=\mt{f_1/t^i\\g_1/t^i}\cdot \mt{f_2/t^i&g_2/t^i}\]
which satisfy $\det A(t)=0 \pmod{t^{n-i+1}}$. Conversely, any rank-one matrix $A(t)\in\Mat_{2\times 2}(k'[t]/(t^{n-i+1}))$ whose constant term is nonzero can be factored in this form by \autoref{lem:mtx-factor}, and two such factorizations differ exactly by multiplication by a unit $u(t)\in \L_{n-i}\mathbb G_m$. Hence, the $\L_{n-i}\mathbb G_m$-action is free and we have an isomorphism
\equa{[(C_{n,i}^0\times& C_{n,i}^0)/\L_{n-i}\mathbb G_m]
\cong\\
&\left\{A(t)\in \operatorname{Mat}_{2\times 2}(k[t]/(t^{n-i+1})):\det A(t)=0,\ A(0)\neq 0\right\}
\cong \L_{n-i}(M).
}
Since $M$ is the punctured affine cone over $\PP^1\times \PP^1$, we have $\e(\L_{n-i}M)=\e(M)\LL^{3n-3i}=(\LL+1)(\LL^2-1)\LL^{3n-3i}$. Therefore
\[\Y^{\cC_{n,i}^0}\cong\L_{n-i}M\times B\GG_a^i\quad\textrm{and}\quad\e(\cC_{n,i}^0)=(\LL+1)(\LL^2-1)\LL^{3n-4i}. \] 

We now turn to the statement about $\cC^2_{n,2i+1}$. We set 
\[C_{n,2i+1}^2:=\left\{(f(t),g(t))=(t\tilde f(t^2),t\tilde g(t^2))\ \middle|\ \min(\val_t(f),\val_t(g))=2i+1\right\}.\]
The $2$-twisted jets corresponding to polynomials in $C_{n,2i+1}^2\times C_{n,2i+1}^2$ induce the embedding $\phi:\mu_2\to \GG_m$ and we need to quotient by $\L^\phi_{m}\GG_m=\L_{m}\GG_m$, which is indeed special. The same argument as above gives 
\[\Y^{\cC_{n,2i+1}^2}\cong \L_{n-i}M\times B\GG_a^i,\qquad \e(\cC_{n,2i+1}^2)=(\LL+1)(\LL^2-1)\LL^{3n-4i}.\] 

Finally, the last claim follows from the fact that $\cC_\infty$ is the single untwisted jet given by the polynomials $x_i=0$, whose stabilizer is the whole group $\L_n\GG_m$. Thus 
\[\Y^{\cC_{n,\infty}^0}=B\L_n\GG_m,\qquad \e(\cC_\infty)=(\LL-1)^{-1}\LL^{-n}.\qedhere\] 
\end{proof} 

\begin{lemma}\label{lem:mtx-factor}
Let $R:=k'[t]/(t^{n+1})$,
and let $A(t)
\in \operatorname{Mat}_{2\times 2}(R)$
be a rank-one matrix such that $A(0)\neq 0$. Then there exist
$f_1,f_2,f_3,f_4\in R$
such that
\[
A(t)=
\begin{pmatrix}
f_1\\g_1
\end{pmatrix}
\begin{pmatrix}
f_2 & g_2
\end{pmatrix}.
\]
\end{lemma}

\begin{proof}
An element
of \(R\) is a unit if and only if its constant term is nonzero. Since $A(0)\neq 0$, at least one entry of $A$ is a unit in $R$. After permuting rows and columns, we assume that
\[
A=\begin{pmatrix}
a & b\\
c & d
\end{pmatrix},\qquad a\in R^\times.
\]
Since $A$ has rank 1, we have $\det(A)=ad-bc=0$ and $d=ca^{-1}b$.
Therefore
\[
A=
\begin{pmatrix}
a & b\\
c & ca^{-1}b
\end{pmatrix}
=
\begin{pmatrix}
a\\
c
\end{pmatrix}
\begin{pmatrix}
1 & a^{-1}b
\end{pmatrix}
\]
is of the desired form.
\end{proof}

We turn now to the main result of this section.

\begin{proof}[{Proof of \autoref{prop:HstrA1-3d}}]
By \autoref{lem:stab}, 
\[H^{*+6n}_c(\cC_{n,i}^{0})(3n)\cong H^{*+6n}_c(\cC_{n,2i+1}^{2})(3n)\cong H^{*+8i}_c(M)(4i)\]
for $w=0$ and $j=0,1,\dots,\floor{\frac{n+1}2}$.
On the other hand, for any fixed $j\in\ZZ$, we have $H^{j+8i}_c(M)=0$ for $i>\floor{\frac{n+1}2}$ and $n>\frac{10-j}{4}$ since $\dim M=3$. By definition of the stringy cohomology, we then have
\[H^{*}_{\str}(\X)=\bigoplus_{i=0}^\infty H^{*+4i}_c(M)(2i).\]
Thus, we only need to compute the compactly supported cohomology of $M$. Note that the total space $L$ of $\mathcal{O}(-1,-1)$ on $B:=\PP^1\times\PP^1$ parametrizes those matrices of rank at most 1, and the complement of the zero section parameterizes the matrices of rank 1. We identify $M$ with $L\setminus B$, and consider the compactly supported long exact sequence
    \[\dots\to H^i_c(M)\to H^i_c(L)\to H^i(B)\to H^{i+1}_c(M)\to \dots.\]
We have $H^*(B,\QQ)=\QQ[a,b]/(a^2,b^2)$ as rings. Then the map $H^{i-2}(B)(-1)\cong H^i_c(L)\to H^i(B)$ is given by the cup product with the Euler class $\e(L)=-(a+b)$. This is injective for $i=2$ and surjective for $i=4$. From this we conclude 
    \[
H_c^i(M) \cong
\begin{cases}
\mathbb{Q},
& j=1 \text{ for some } i\geq 0,\\
\mathbb{Q}(-1),
& j=3 \text{ for some } i\geq 0,\\
\mathbb{Q}(-2),
& j=4 \text{ for some } i\geq 0,\\
\mathbb{Q}(-3),
& j=6 \text{ for some } i\geq 0,\\
0,
& \text{otherwise}.
\end{cases}
\]
\end{proof}

Observe that the non-trivial Hodge numbers of $M$ are
\[h^{3,3}(H^{6}_c(M))=h^{2,2}(H^{4}_c(M))=h^{1,1}(H^{3}_c(M))=h^{0,0}(H^{1}_c(M))=1.\]
This implies that $\HD(H^*_c(M))=(uv)^3+(uv)^2-uv-1$. 
Hence 
\[
\HD(H^*_{\str}(\X))=\sum_{i=0}^\infty (uv)^{-2i}((uv)^3+(uv)^2-uv-1)=(uv)^3+(uv)^2.
\]
On the other hand, we may compute the stringy $E$-function of $Y$ using toric geometry. Since the fan of $Y$ is the cone over a square, we obtain a toric small resolution $f:X\to Y$ by triangulating the cone. The fan of $X$ has $1$ zero-dimensional cone, $4$ rays, $5$ two-dimensional cones, and $2$ three-dimensional cones. This shows
\[\e(X)=\sum_{\sigma\in\Sigma}(\LL-1)^{3-\dim\sigma}=(\LL-1)^3+4(\LL-1)^2+5(\LL-1)+2=\LL^3+\LL^2.\]
Therefore we see $\HD_{\str}(Y)=\HD_{\str}(\cX)$ as was shown more generally in \autoref{corollaryStringyCohomologyCrepantResolution}.

\bibliographystyle{alpha}
\bibliography{CISHN}

\begin{thebibliography}{{Eke}09}

\bibitem[Bat98]{Batyrev}
Victor~V. Batyrev.
\newblock Stringy {H}odge numbers of varieties with {G}orenstein canonical
  singularities.
\newblock In {\em Integrable systems and algebraic geometry ({K}obe/{K}yoto,
  1997)}, pages 1--32. World Sci. Publ., River Edge, NJ, 1998.

\bibitem[BD96]{BatyrevDais}
Victor~V. Batyrev and Dimitrios~I. Dais.
\newblock Strong {M}c{K}ay correspondence, string-theoretic {H}odge numbers and
  mirror symmetry.
\newblock {\em Topology}, 35(4):901--929, 1996.

\bibitem[BM13]{BatyrevMoreau}
Victor Batyrev and Anne Moreau.
\newblock The arc space of horospherical varieties and motivic integration.
\newblock {\em Compos. Math.}, 149(8):1327--1352, 2013.

\bibitem[Bor14]{Borisov}
Lev~A. Borisov.
\newblock On stringy cohomology spaces.
\newblock {\em Duke Math. J.}, 163(6):1105--1126, 2014.

\bibitem[Bri22]{BrionHomAlg}
Michel Brion.
\newblock Homomorphisms of algebraic groups: representability and rigidity.
\newblock {\em Michigan Math. J.}, 72:51--76, 2022.

\bibitem[CR04]{ChenRuan}
Weimin Chen and Yongbin Ruan.
\newblock A new cohomology theory of orbifold.
\newblock {\em Comm. Math. Phys.}, 248(1):1--31, 2004.

\bibitem[{Eke}09]{Ekedahl}
Torsten {Ekedahl}.
\newblock {The Grothendieck group of algebraic stacks}.
\newblock {\em arXiv e-prints}, page arXiv:0903.3143, March 2009.

\bibitem[ER21]{EdidinRydh}
Dan Edidin and David Rydh.
\newblock Canonical reduction of stabilizers for {A}rtin stacks with good
  moduli spaces.
\newblock {\em Duke Math. J.}, 170(5):827--880, 2021.

\bibitem[HR17]{HallRydh}
Jack Hall and David Rydh.
\newblock The telescope conjecture for algebraic stacks.
\newblock {\em J. Topol.}, 10(3):776--794, 2017.

\bibitem[Ley01]{Leykin}
Anton Leykin.
\newblock Constructibility of the set of polynomials with a fixed
  {B}ernstein-{S}ato polynomial: an algorithmic approach.
\newblock volume~32, pages 663--675. 2001.
\newblock Effective methods in rings of differential operators.

\bibitem[Ryd11]{Rydh}
David Rydh.
\newblock Representability of {H}ilbert schemes and {H}ilbert stacks of points.
\newblock {\em Comm. Algebra}, 39(7):2632--2646, 2011.

\bibitem[{Sta}18]{stacks-project}
The {Stacks Project Authors}.
\newblock \textit{Stacks Project}.
\newblock \url{https://stacks.math.columbia.edu}, 2018.

\bibitem[SU22]{SatrianoUsatine1}
Matthew Satriano and Jeremy Usatine.
\newblock Stringy invariants and toric {A}rtin stacks.
\newblock {\em Forum Math. Sigma}, 10:Paper No. e9, 60, 2022.

\bibitem[SU23a]{SatrianoUsatine4}
Matthew {Satriano} and Jeremy {Usatine}.
\newblock {Beyond twisted maps: applications to motivic integration}.
\newblock {\em arXiv e-prints}, page arXiv:2309.11434, September 2023.

\bibitem[SU23b]{SatrianoUsatine5}
Matthew {Satriano} and Jeremy {Usatine}.
\newblock {Motivic integration for singular Artin stacks}.
\newblock {\em arXiv e-prints}, page arXiv:2309.11442, September 2023.

\bibitem[SU24a]{SatrianoUsatine3}
Matthew Satriano and Jeremy Usatine.
\newblock Crepant resolutions of log-terminal singularities via {A}rtin stacks.
\newblock {\em Duke Math. J.}, 173(18):3605--3646, 2024.

\bibitem[SU24b]{SatrianoUsatine6}
Matthew {Satriano} and Jeremy {Usatine}.
\newblock {Stringy Hodge numbers via crepant resolutions by Artin stacks}.
\newblock {\em arXiv e-prints}, page arXiv:2410.23951, October 2024.

\bibitem[SV07]{SchepersVeys}
Jan Schepers and Willem Veys.
\newblock Stringy {H}odge numbers for a class of isolated singularities and for
  threefolds.
\newblock {\em Int. Math. Res. Not. IMRN}, (2):Art. ID rnm016, 14, 2007.

\bibitem[Tub25a]{Tubach2}
Swann Tubach.
\newblock Mixed {H}odge modules on stacks.
\newblock {\em Forum Math. Sigma}, 13:Paper No. e175, 29, 2025.

\bibitem[Tub25b]{Tubach1}
Swann Tubach.
\newblock On the {N}ori and {H}odge realisations of {V}oevodsky motives.
\newblock {\em Compos. Math.}, 161(9):2155--2201, 2025.

\bibitem[Vak25]{Vakil}
Ravi Vakil.
\newblock {\em The rising sea---foundations of algebraic geometry}.
\newblock Princeton University Press, Princeton, NJ, [2025] \copyright 2025.

\bibitem[VdB23]{vandenBergh}
Michel Van~den Bergh.
\newblock Noncommutative crepant resolutions, an overview.
\newblock In {\em I{CM}---{I}nternational {C}ongress of {M}athematicians.
  {V}ol. 2. {P}lenary lectures}, pages 1354--1391. EMS Press, Berlin, [2023]
  \copyright 2023.

\bibitem[Yas04]{Yasuda2004}
Takehiko Yasuda.
\newblock Twisted jets, motivic measures and orbifold cohomology.
\newblock {\em Compos. Math.}, 140(2):396--422, 2004.

\bibitem[Yas06]{Yasuda}
Takehiko Yasuda.
\newblock Motivic integration over {D}eligne-{M}umford stacks.
\newblock {\em Adv. Math.}, 207(2):707--761, 2006.

\end{thebibliography}

\end{document}